\RequirePackage{etoolbox}
\csdef{input@path}{{style/}{graphics/}}
\documentclass[MSNbibl,number,citesort,dvips]{arxbj}
\makeatletter
   \@ifpackageloaded{graphicx}{}{\usepackage{graphicx}}
\makeatother

\aid{0}
\volume{21}
\issue{2}
\pubyear{2015}
\firstpage{1260}
\lastpage{1288}
\doi{10.3150/14-BEJ606} 

\makeatletter

\newcommand{\rright}{\right}
\newcommand{\lleft}{\left}
\newtheorem{teo}{Theorem}
\newtheorem{lem}[teo]{Lemma}
\newtheorem{cor}[teo]{Corollary}
\newtheorem{prop}[teo]{Proposition}
\newproclaim{defn}[teo]{Definition}
\newremark{rem}[teo]{Remark}
\newremark{Notation}{Notation}

\newcommand{\RR}{\mathbb{R}}
\newcommand{\NN}{\mathbb{N}}
\newcommand{\CLQ}{\mathcal{Q}}
\newcommand{\CLW}{\mathcal{W}}
\newcommand{\osc}{\mathrm{O}\mathit{sc}}
\renewcommand{\epsilon}{\varepsilon}
\renewcommand{\Pr}{\operatorname{Pr}}

\newcommand{\med}{\operatorname{med}}

\def\sfrac#1#2{#1/#2}
\def\vfrac#1#2{(#1)/#2}
\def\afrac#1#2{#1/(#2)}

\makeatother

\begin{document}
\begin{frontmatter}

\title{Weighted power variation of integrals with respect to a Gaussian
process}
\runtitle{Power variation of integrals}

\begin{aug}
\author[A]{\inits{R.}\fnms{Rimas}~\snm{Norvai\v sa}\corref{}\thanksref{A}\ead[label=e1]{rimas.norvaisa@mii.vu.lt}}
\address[A]{Institute of Mathematics and Informatics, Vilnius
University, Akademijos 4, Vilnius, Lithuania.\\
\printead{e1}}
\end{aug}

\received{\smonth{5} \syear{2011}}
\revised{\smonth{4} \syear{2013}}

%
\begin{abstract}
We consider a stochastic process $Y$ defined by an integral in
quadratic mean of a deterministic function
$f$ with respect to a Gaussian process $X$, which need not have
stationary increments.
For a class of Gaussian processes $X$, it is proved that sums of
properly weighted powers of increments
of $Y$ over a sequence of partitions of a time interval converge almost surely.
The conditions of this result are expressed in terms of the
$p$-variation of the covariance function
of $X$.
In particular, the result holds when $X$ is a fractional Brownian
motion, a subfractional Brownian motion
and a bifractional Brownian motion.
\end{abstract}

%

\end{frontmatter}

\section{Introduction}

Let $X=\{X(t)\dvtx  t\in[0,T]\}$ be a Gaussian process and let
$f\dvtx [0,T]\to\RR$ be a real-valued function for some $0<T<\infty$.
We consider a stochastic process $Y=\{Y(t)\dvtx  t\in[0,T]\}$,
given by an integral
%
\begin{equation}
\label{Y1} Y(t)=\mathrm{q.m.}\int_0^tf \,\mathrm{d}X,\qquad  0\leq t
\leq T,
\end{equation}
defined as a limit of Riemann--Stieltjes sums converging in \emph
{quadratic mean}.
According to the main result of this paper (Theorem~\ref{asconvergence}),
under suitable hypotheses on the covariance of $X$ and
the $p$-variation of $f$, there exists a stochastic process $Y$ defined
by (\ref{Y1})
and with probability one
%
\begin{equation}
\label{mainresult} \lim_{n\to\infty}\sum_{i=1}^{m_n}
\frac{|Y(t_i^n)-Y(t_{i-1}^n)|^r}{[
\rho(t_i^n-t_{i-1}^n)]^r}\bigl(t_i^n-t_{i-1}^n
\bigr)=E|\eta|^r\int_0^T|f|^r,
\end{equation}
where $\eta$ is a standard normal random variable, 
$\rho$ is a function equivalent to
$(E[X(s+\cdot)-X(s)]^2)^{1/2}$ near zero uniformly in $s\in[\epsilon,T)$
for each $\epsilon>0$, and $((t_i^n)_{i=0}^{m_n})$ is a sequence
of partitions of $[0,T]$ such that the sequence $(\max_i(t_i^n-t_{i-1}^n))$
tends to zero as $n\to\infty$ sufficiently fast.

In the case $f\equiv1$, $X$ is a real centered Brownian motion $B$,
$\rho(h)=\sqrt{h}$, $T=1$ and $t_i^n=i2^{-n}$, $i\in\{0,1,\dots
,2^n\}$,
(\ref{mainresult}) gives the result of L\'evy \cite{PL1940}
%
\begin{equation}
\label{PL} \lim_{n\to\infty}\sum_{i=1}^{2^n}
\bigl[B \bigl(i2^{-n} \bigr)-B \bigl((i-1)2^{-n} \bigr)
\bigr]^2=1 \qquad \mbox{a.s.}
\end{equation}
This result has been extended in many directions.
The Brownian motion $B$ has been replaced by a general Gaussian procces $X$
under suitable hypotheses on the covariance of $X$ (\cite{GB56,EGG61,RN2010}).
The sequence of dyadic partitions of $[0,1]$ in (\ref{PL}) have been replaced
by a sequence $((t_i^n)_{i=0}^{m_n})$ of partitions of $[0,1]$ such that
$\max_i(t_i^n-t_{i-1}^n)=\mathrm{o}(1/\log n)$ as $n\to\infty$ (\cite{RMD73}),
and this is the best possible rate (\cite{V1974}).
The second power of increments of a Brownian motion $B$ in (\ref{PL})
has been
replaced by $r$th power of increments of a Gaussian process $X$ with stationary
increments by Marcus and Rosen \cite{MandR1992} when $r\geq2$ and
by Shao \cite{QMS1996} when $r>1$.
A different strand of research led to similar results for a power
variation of
a general stochastic process with convergence in probability in place of
convergence with probability one (see  \cite
{CNW2006}, and references there).
In fact, the present paper is an attempt to prove the almost sure convergence
for a weighted power variation of integral process (\ref{Y1}) like in
\cite{CNW2006}
(but different) and keeping the framework of the above mentioned results.

In the rest of this section, we formulate and discuss in more detail
the stated main result
of the paper.
In Section~\ref{sec2}, we give conditions for the existence of the integral
process $Y$
given by (\ref{Y1}).
In Section~\ref{covariance}, we give conditions for a Gaussian process to have
positively or negatively
correlated increments in terms of $p$-variation of its covariance function.
The general results are proved to hold for a fractional Brownian
motion, a subfractional Brownian
motion and a bifractional Brownian motion.
The main result is proved in Section~\ref{sec4}.

We consider the integral process $Y$ given by (\ref{Y1}) with
a Gaussian process $X$ having ``locally stationary increments'' defined next
(Definition~\ref{LSI}).
While in this paper we consider only Gaussian processes, the following concept
makes sense for any stochastic process whose finite-dimensional distributions
have finite moments of the first and second orders.
Such a process with mean zero will be called a
second order stochastic process as in \cite
{ML1978}, Chapter~37.

\begin{defn}\label{LV}
Let $T>0$ and let $R[0,T]$ be a set of functions
$\rho\dvtx [0,T]\to\RR_+$ such that $\rho(0)=0$, $\rho$ is
continuous at zero,
and for each $\delta\in(0,T)$,
%
\begin{equation}
\label{A0} 0<\inf\bigl\{\rho(u)\dvtx  u\in[\delta,T]\bigr\} \leq\sup\bigl\{
\rho(u)\dvtx  u\in[\delta,T]\bigr\}<\infty.
\end{equation}
Let $X=\{X(t)\dvtx  t\in[0,T]\}$ be a second order stochastic process
with the incremental variance function $\sigma_X^2$ defined on
$[0,T]^2:=[0,T]\times[0,T]$
with values
\[
\sigma_X^2(s,t):=E\bigl[X(t)-X(s)\bigr]^2,\qquad
(s,t)\in[0,T]^2.
\]
We say that $X$ has a \emph{local variance} if there is a function
$\rho\in R[0,T]$
such that \textup{(A1)} and \textup{(A2)} hold, where
\begin{enumerate}[(A2)]
\item[(A1)] there is a finite constant $L$ such that for all
$(s,t)\in[0,T]^2$
\[
\sigma_X(s,t)\leq L\rho\bigl(|t-s|\bigr) ;
\]
\item[(A2)] for each $\epsilon\in(0,T)$
%
\begin{equation}
\label{A2} \lim_{\delta\downarrow0}\sup \biggl\{ \biggl|\frac{\sigma
_X(s,s+h)}{\rho(h)}-1 \biggr|
\dvtx  s\in[\epsilon,T), h\in\bigl(0,\delta\wedge(T-s)\bigr] \biggr\}=0.
\end{equation}
\end{enumerate}
In this case, we say that $X$ has a local variance with $\rho\in R[0,T]$.
\end{defn}

Let $X=\{X(t)\dvtx  t\in[0,T]\}$, $T>0$, be a mean zero Gaussian
process with stationary increments,
and let $\rho_X(u):=\sigma_X(u,0)$ for each $u\in[0,T]$.
Then (A1) and (A2) for $\rho=\rho_X$ hold trivially.
Also, if $X$ is such that $\rho_X$ is continuous at zero, and (\ref
{A0}) holds
for $\rho=\rho_X$ and each $\delta\in(0,T)$, then $\rho_X\in
R[0,T]$, and so
$X$ has a local variance with $\rho_X$.

Suppose $X$ is a second order stochastic process such that $\sigma
_X(s,t)\neq0$
for each $(s,t)\in[0,T]^2$.
If $X$ has a local variance with two elements $\rho_1$ and $\rho_2$
in $R[0,T]$,
then by (A2) we have
\[
\lim_{u\downarrow0}\rho_1(u)/\rho_2(u)=1.
\]
This property defines a binary relation in the set $R[0,T]$, which is
an equivalence relation.
Let us denote this relation by $\sim$.
If $X$ has a local variance with $\rho_1\in R[0,T]$, and if $\rho
_2\in R[0,T]$ is
such that $\rho_1\sim\rho_2$, then $X$ has a local variance with
$\rho_2$.
Therefore, the property of $X$ having a local variance is a class
invariant under
the binary relation $\sim$.

\begin{defn}\label{LSI}
Let $X$ be a second order stochastic process.
We say that $X$ has \emph{locally stationary increments} if $X$ has a
local variance
with some $\rho\in R[0,T]$.
Any element in the equivalence class $\{\rho'\in R[0,T]\dvtx \rho
'\sim\rho\}$
will be called a \emph{local variance function}.
We write $X\in\mathcal{LSI}(\rho(\cdot))$ if $X$ has local variance
with $\rho\in R[0,T]$.
\end{defn}

So far as we are aware, a similar concept was suggested by \cite{SB1974}, Section~8,
under the name of \emph{local stationarity}.
We show that a subfractional Brownian motion $G_H=\{G_H(t)\dvtx  t\in
[0,T]\}$
with index $H\in(0,1)$ and covariance function (\ref{sfBm1}) has
local variance function
$\rho_H(u)=u^H$, $u\in[0,T]$ (see Proposition~\ref{sfBm3}).
Also, we show that a bifractional Brownian motion $B_{H,K}=\{
B_{H,K}(t)\dvtx  t\in[0,T]\}$
with parameters $(H,K)\in(0,1)\times(0,1)$ and covariance function
(\ref{bfBm}) has local
variance function $\rho_{H,K}(u)=2^{(1-K)/2}u^{HK}$, $u\in[0,T]$ (see
Proposition~\ref{bfBm2}).

For $\rho\in R[0,T]$ let
\[
\gamma_{\ast}(\rho):=\inf\bigl\{\gamma>0\dvtx  u^{\gamma}/\rho(u)
\to 0 \mbox{ as } u\downarrow0\bigr\}=\limsup_{u\downarrow0}
\frac{\log\rho(u)}{\log u}
\]
and
\[
\gamma^{\ast}(\rho):=\sup\bigl\{\gamma>0\dvtx  u^{\gamma}/\rho(u)
\to +\infty\mbox{ as } u\downarrow0\bigr\}=\liminf_{u\downarrow0}
\frac{\log\rho(u)}{\log u}.
\]
By definition, we have $0\leq\gamma^{\ast}(\rho)\leq\gamma_{\ast
}(\rho)\leq\infty$.
Clearly, $\gamma_{\ast}(\rho)$ and $\gamma^{\ast}(\rho)$ do not
change when $\rho$
is replaced by $\rho'\sim\rho$.
If a second order stochastic process $X$ has a local variance with
$\rho\in R[0,T]$
and if $0<\gamma^{\ast}(\rho)=\gamma_{\ast}(\rho)<\infty$ then
we will say
that $X$ has the \emph{Orey index} $\gamma_X:=\gamma^{\ast}(\rho
)=\gamma_{\ast}(\rho)$.
Clearly this notion extends the one suggested by Orey
\cite{SO1970} (see
also \cite{NandS2000})
for a Gaussian stochastic process with stationary increments.
We will be interested in the case in which a Gaussian stochastic process
$X\in\mathcal{LSI}(\rho(\cdot))$ has the Orey index $\gamma_X\in(0,1)$.
In this case, $X$ is equivalent to a stochastic process whose almost
all sample functions
satisfy a H\"older condition of order $\alpha$ for each $\alpha
<\gamma_X$.

Now we can formulate the main result of the paper with more details.
Suppose that a mean zero Gaussian process $X$ has locally stationary increments
with a local variance $\rho$ and has the Orey index $\gamma_X(\rho
)=\gamma\in(0,1)$.
Suppose that a function $f\dvtx [0,T]\to\RR$ is regulated if $\gamma
\geq1/2$
and has bounded $q$-variation for some $q<1/(1-2\gamma)$ if $\gamma<1/2$,
and let $1<r<2/\max\{(2\gamma-1),0\}$.
Under the further hypotheses of Theorem~\ref{asconvergence} on the
covariance of $X$,
a stochastic process $Y$ defined by (\ref{Y1}) exists, and
(\ref{mainresult}) holds with probability one.
The proof of the main result (Theorem~\ref{asconvergence}) use the
ideas of Marcus and Rosen \cite{MandR1992}
and Shao \cite{QMS1996}.

Gladyshev \cite{EGG61} considered a stochastic process $X=\{
X(t)\dvtx  t\in
[0,1]\}$
with Gaussian increments, mean zero and a covariance function $\Gamma
_X$ such that
the expression
%
\begin{equation}
\label{EEG1} \sigma_X^2(t,t-h)/h^{2\gamma}= \bigl[
\Gamma_X(t,t)-2\Gamma_X(t,t-h)+\Gamma_X(t-h,t-h)
\bigr]/h^{2\gamma}
\end{equation}
converges uniformly to a function $g$ on $[0,1]$ as $h\to0$,
$\Gamma_X$ is continuous, twice differentiable outside the diagonal and
%
\begin{equation}
\label{EEG2} \biggl|\frac{\partial^2\Gamma_X(t,s)}{\partial t\,\partial s} \biggr| \leq\frac{C}{|t-s|^{2(1-\gamma)}}
\end{equation}
(here $\gamma=1-\tilde{\gamma}/2$ for $\tilde{\gamma}$ in
\cite{EGG61}).
Under these assumptions, E. G. Gladyshev proved (\ref{mainresult})
with the right-hand
side replaced by $\int_0^1g$ when $f\equiv1$,
$\rho(u)=u^{\gamma}$, $r=2$, $t_i^n=i2^{-n}$ for $i\in\{1,\dots
,2^n\}$ and
each $n\geq1$.
In \cite{RN2010}, we showed that hypothesis (\ref
{EEG1}) does not hold
when $X$ is a subfractional Brownian motion and a bifractional Brownian motion,
but the conclusion of Theorem~1 in \cite{EGG61}
(with $g\equiv1$)
still holds for these processes.
Malukas \cite{RM2011} further extended this result to arbitrary
sequences of partitions using the ideas of Klein and Gin\'e \cite{KandG},
and proved a central limit theorem in his setting.

As compared to previous results, in the present paper a class of
Gaussian processes
is defined by conditions (A1) and (A2) which seem to fit perfectly
Gladyshev's theorem
for the mean convergence (see Corollary~\ref{meanX} below), and are
weaker than hypothesis
(\ref{EEG1}) with $g\equiv1$.
Instead of hypothesis (\ref{EEG2}), we use the following assumption on
a Gaussian process
$X$ having locally stationary increments and the Orey index $\gamma\in(0,1)$:
there is a constant $C_2$ such that the inequality
\[
\sum_{j=1}^{m} \bigl|E
\bigl[X(t_i)-X(t_{i-1})\bigr] \bigl[X(t_j)-X(t_{j-1})
\bigr] \bigr| \leq C_2 (t_i-t_{i-1})^{1\wedge(2\gamma)}
\]
holds for each partition $(t_j)_{j\in\{0,\ldots,m\}}$ of $[0,T]$ and each
$i\in\{1,\dots,m\}$
(see Corollary~\ref{asconvergenceX}).
Finally, in place of $X$, we consider a stochastic process $Y$
defined by (\ref{Y1}).
In this case the preceding assumption on $X$ is replaced by the
following one:
there is a constant $C_2$ such that the inequality
\[
\sum_{j=1}^mV_p\bigl(
\Gamma_X;[t_{i-1},t_i]\times[t_{j-1},t_j]
\bigr)\leq C_2(t_i-t_{i-1})^{1\wedge(2\gamma)}
\]
holds for each partition $(t_j)_{j\in\{0,\ldots,m\}}$ of $[0,T]$ and each
$i\in\{1,\dots,m\}$,
where $V_p(\cdot)$ is the $p$-variation seminorm defined by (\ref
{Vp2}) below
and $p=\max\{1,1/(2\gamma)\}$ (see Theorem~\ref{asconvergence}).
The two assumptions are shown to be easily verified using the
properties of negative or positive
correlation of $X$ (see Section~\ref{covariance}).

The following is a consequence of Theorem~\ref{asconvergence},
Proposition~\ref{fBm1}
when $K=1$ and Proposition~\ref{bfBm2} when $K\in(0,1)$.

%
\begin{cor}\label{bfBm3}
Let $T>0$, $H\in(0,1)$, $K\in(0,1]$, $r\in(1,2/\max\{(2HK-1),0\})$,
and let $B_{H,K}=\{B_{H,K}(t)\dvtx  t\in[0,T]\}$ be
a bifractional Brownian motion with parameters $(H,K)$.
Let $f\dvtx [0,T]\to\RR$ be regulated if $HK\geq1/2$ and of bounded
$q$-variation for
some $q<1/(1-2HK)$ if $HK<1/2$.
Let $(\kappa_n)$ be a sequence of partitions $\kappa_n=(t_i^n)_{i\in
\{0,\dots,m_n\}}$
of $[0,T]$ such that
\[
\lim_{n\to\infty}|\kappa_n|^{(1\wedge\sfrac{2}{r})+(0\wedge
(1-2HK))}\log n=0.
\]
Then with probability one
\[
\lim_{n\to\infty}\sum_{i=1}^{m_n}
\biggl|\mathrm{q.m.}\int_{t_{i-1}^n}^{t_i^n}f \,\mathrm{d}B_{H,K}
\biggr|^r \bigl(t_i^n-t_{i-1}^n
\bigr)^{1-rHK}=2^{r(1-K)/2}E|\eta|^r\int
_0^T|f|^r,
\]
where $\eta$ is a standard normal random variable.
\end{cor}

A similar result holds for a subfractional Brownian motion due to
Theorem~\ref{asconvergence}
and Proposition~\ref{sfBm3}.
Corollary~\ref{bfBm3} when $K=1$ (the case of fractional Brownian
motion $B_H$) may be compared with
Theorem~1 of  \cite{CNW2006} where
$f$ is a stochastic process, the integral $\int_0^tf \,\mathrm{d}B_H$, $t\in
[0,1]$, is
defined pathwise as the Riemann--Stieltjes integral,
partition $\kappa_n=(i/n)_{i\in\{0,\dots,n\}}$, convergence holds in
probability and
with no restrictions on $r$.

\begin{Notation*}
For $n\in\NN:=\{0,1,\dots\}$ let $[n]:=\{0,1,\dots,n\}$ and $(n]:=\{
1,\dots,n\}$.
An interval $[a,b]$ is a closed set of real numbers $r$ such that
$a\leq r\leq b$.
A partition of an interval $[a,b]$ is a finite sequence of real numbers
$\kappa=(t_i)_{i\in[n]}$
such that $a=t_0<t_1<\cdots<t_n=b$.
The set of all partitions of $[a,b]$ is denoted by $\Pi[a,b]$.
Given a partition $\kappa=(t_i)_{i\in[n]}$, for each $i\in(n]$, let
$J_i^{\kappa}:=[t_{i-1},t_i]$ and
$\Delta_i^{\kappa}:=t_i-t_{i-1}$.
The mesh of a partition $\kappa$ is $|\kappa|:=\max_i\Delta
_i^{\kappa}$.
Given a function $g\dvtx [a,b]\to\RR$ and a sequence $(\kappa_n)$
of partitions
$\kappa_n=(t_i^{n})_{i\in[m_n]}$ of $[a,b]$, for each $i\in(m_n]$,
let $\Delta_i^n:=\Delta_i^{\kappa_n}=t_i^n-t_{i-1}^n$ and $\Delta
_i^ng:=g(t_i^n)-g(t_{i-1}^n)$.
\end{Notation*}

\section{Riemann--Stieltjes integrals}\label{sec2}

In this section, the double Riemann--Stieltjes integral and the
quadratic mean
Riemann--Stieltjes integral are defined, and several their properties to
be used are given.

\subsection*{A double Riemann--Stieltjes integral}
Let $F$ and $G$ be real-valued functions defined on
a rectangle $R:=[a,b]\times[c,d]$ in $\RR^2$ defined by real numbers
$a<b$ and
$c<d$.
We recall a definition of the Riemann--Stieltjes integral of $F$ with
respect to $G$ over $R$.
A partition of $[a,b]\times[c,d]$ is a finite double sequence of pairs
of real numbers $\tau=\{(s_i,t_j)\dvtx (i,j)\in[n]\times[m]\}$
such that $(s_i)_{i\in[n]}\in\Pi[a,b]$ and $(t_j)_{j\in[m]}\in\Pi[c,d]$.
The set of all partitions of a rectangle $R$ is denoted by $\Pi(R)$.
Thus $\tau\in\Pi(R)$ if and only if $\tau=\kappa\times\lambda$
for some
$\kappa\in\Pi[a,b]$ and $\lambda\in\Pi[c,d]$.
The \emph{mesh} of $\tau=\kappa\times\lambda\in\Pi(R)$ is
$|\tau|:=\max\{|\kappa|,|\lambda|\}$.
Given such $\tau$, for each $i\in(n]$ and $j\in(m]$,
the double increment of $G$ over the rectangle
$Q_{i,j}=[s_{i-1},s_i]\times[t_{j-1},t_j]$ is defined by
%
\begin{equation}
\label{dincrement} \Delta_{i,j}^{\tau}G:=\Delta^{Q_{i,j}}G
:=G(s_i,t_j)-G(s_{i-1},t_j)-G(s_i,t_{j-1})
+G(s_{i-1},t_{j-1}).
\end{equation}
Also if $(u_i,v_j)\in Q_{i,j}$ for $(i,j)\in(n]\times(m]$, then
$(u_i,v_j)$ is called a \emph{tag} and
the collection $\dot{\tau}:=\{((u_i,v_j),Q_{i,j}) \dvtx (i,j)\in
(n]\times(m]\}$
is called a \emph{tagged partition} of $R$.
The \emph{Riemann--Stieltjes sum} of $F$ with respect to
$G$ and based on a tagged partition $\dot{\tau}$ is
\[
S_{\mathrm{RS}}\bigl(F,\Delta^2G;\dot{\tau}\bigr):=\sum
_{i=1}^n\sum_{j=1}^m
F(u_i,v_j)\Delta_{i,j}^{\tau}G.
\]

We say that the \emph{double Riemann--Stieltjes} integral over
$[a,b]\times[c,d]$ of $F$ with respect to $G$ exists and equals
$A\in\RR$, if for each $\epsilon>0$ there is a $\delta>0$ such that
\[
\bigl|S_{\mathrm{RS}}\bigl(F,\Delta^2G;\dot{\tau}\bigr) -A \bigr|<\epsilon
\]
for each tagged partition $\dot{\tau}$ of $[a,b]\times[c,d]$
with the mesh $|\tau|<\delta$.
Clearly, if such $A$ exists then it is unique and is denoted by
\[
\int_a^b\int_c^dF
\,\mathrm{d}^2G =\int_a^b\int
_c^d F(s,t) \,\mathrm{d}^2G(s,t):=A.
\]
Since in this paper we work with the quadratic mean Riemann--Stieltjes
integral $\int f \,\mathrm{d}X$ of a deterministic function $f$ it is enough
to treat double Riemann--Stieltjes integral for integrands $F=f\otimes f$,
where $f\otimes f(s,t)=f(s)f(t)$ for $(s,t)\in R$.

First, we give sufficient conditions for the existence of a double
Riemann--Stieltjes
when the integrator has bounded total variation.
Let $R=[a,b]\times[c,d]$ be a rectangle and $G\dvtx  R\to\RR$.
For a partition $\tau=\{(s_i,t_j)\dvtx (i,j)\in[n]\times[m]\}\in
\Pi(R)$
let
\[
s_1(G;\tau):=\sum_{i=1}^n
\sum_{j=1}^m\bigl|\Delta_{i,j}^{\tau}G\bigr|,
\]
where $\Delta_{i,j}^{\tau}G$ is defined by (\ref{dincrement}).

\begin{defn}\label{Vitali}
Let $R$ be a rectangle in $\RR^2$ and $G\dvtx  R\to\RR$.
\[
V_1(G;R):=\sup\bigl\{s_1(G;\tau)\dvtx \tau\in\Pi(R)\bigr
\}.
\]
If $V_1(G,R)<\infty$ then one says that $G$ is of bounded variation
in the sense of Vitali--Lebesgue--Fr\'echet--de la Vall\'ee Poussin
and write $G\in\CLW_1 (R)$.
\end{defn}

We say that a function $G\dvtx  R\to\RR$ is \emph{separately
continuous} if
its sections $x\mapsto G(x,y)$ and $y \mapsto G(x,y)$ are continuous
for each fixed $y$ and $x$, respectively.
A function $f\dvtx [a,b]\to\RR$ is regulated if for each $x\in
(a,b]$ it
has left limits $f(x-)$ and for each $x\in[a,b)$ it has right limits
$f(x+)$.
The set of all regulated functions on $[a,b]$ is denoted by $\CLW
_{\infty}[a,b]$.
Each regulated function is bounded and for such a function $f$ we write
\[
\|f\|_{\sup}:=\sup\bigl\{\bigl|f(x)\bigr|\dvtx  x\in[a,b]\bigr\},\qquad  \osc(f):=\sup
\bigl\{\bigl|f(x)-f(y)\bigr|\dvtx  x,y\in[a,b]\bigr\}.
\]

\begin{teo}\label{DRS}
Let $R=[a,b]\times[c,d]$ for some real numbers $a<b$ and $c<d$.
Let $G\in\CLW_1 (R)$ be separately continuous, $f\in\CLW_{\infty
}[a,b]$ and
$g\in\CLW_{\infty}[c,d]$.
Then the double Riemann--Stieltjes integral
$\int_a^b\int_c^df\otimes g \,\mathrm{d}^2G$ is defined and
we have the bounds
%
\begin{eqnarray}
\label{2DRS} \biggl|\int_a^b\int
_c^df\otimes g \,\mathrm{d}^2G \biggr|&\leq& \|f
\|_{\sup}\|g\|_{\sup}V_1(G;R),
\\
\label{1DRS} \biggl|\int_a^b\int
_c^d\bigl[f\otimes g-f(a)g(c)\bigr]
\,\mathrm{d}^2G \biggr| &\leq& \bigl[\|g\|_{\sup}\osc(f)+\|f\|_{\sup}
\osc(g) \bigr]V_1(G;R).
\end{eqnarray}
\end{teo}

The proof is standard for such statements about existence of Riemann--Stieltjes
integrals when the integrand is a regulated function and the integrator
has bounded variation
(see, e.g., Theorem~2.17 in \cite{DandN2010}
when functions have
single variable).
Namely, one needs to compare a difference between two Riemann--Stieltjes sums
corresponding to sufficiently fine partitions and one of them is a refinement
of the other.
The sum of terms corresponding to subrectangles containing a jump of
either $f$ or $g$
can be made small due to separate continuity of $G$ and since $G$ is a
difference of two quasi-monotone functions as shown in \cite{EWH1927}, page 345.
The details are omitted.

Next, we give sufficient conditions for the existence of a double
Riemann--Stieltjes
integral in terms of $p$-variation of the integrand and integrator.
Let $p\geq1$ and let $f\dvtx [a,b]\to\RR$.
For a partition $\kappa=(s_i)_{i\in[n]}$ of $[a,b]$, let
\[
s_p(f;\kappa):=\sum_{i=1}^n\bigl|f(s_i)-f(s_{i-1})\bigr|^p.
\]
The $p$-variation seminorm of $f$ on $[a,b]$ is the quantity
\[
V_p(f)=V_p\bigl(f;[a,b]\bigr):=\sup\bigl\{
\bigl[s_p(f;\kappa)\bigr]^{1/p}\dvtx \kappa\in\Pi [a,b]\bigr
\}.
\]
One says that $f$ has bounded $p$-variation or $f\in\CLW_p[a,b]$ if
$V_p(f;[a,b])<\infty$.
We also use the $p$-variation norm defined by
\[
\|f\|_{[p]}=\|f\|_{[p],[a,b]}:=\|f\|_{\sup}+V_p
\bigl(f;[a,b]\bigr).
\]
Recalling that $\CLW_{\infty}[a,b]$ is the set of regulated functions
on $[a,b]$,
$\CLW_p[a,b]$ is defined for $1\leq p\leq\infty$.

Again, let $R=[a,b]\times[c,d]$ be a rectangular and $G\dvtx  R\to\RR$.
For $p\geq1$ and a partition $\tau=\{(s_i,t_j)\dvtx (i,j)\in
[n]\times[m]\}$ of $R$ let\vspace*{-3pt}
\[
s_p(G;\tau):=\sum_{i=1}^n
\sum_{j=1}^m\bigl|\Delta_{i,j}^{\tau}G\bigr|^p,
\]
where $\Delta_{i,j}^{\tau}G$ is defined by (\ref{dincrement}).
The $p$-variation seminorm of $G$ is\vspace*{-3pt}
%
\begin{equation}
\label{Vp2} V_p(G;R):=\sup\bigl\{\bigl[s_p(G;\tau)
\bigr]^{1/p}\dvtx  \tau\in\Pi(R)\bigr\}.
\end{equation}
Let $\CLW_{p}(R)$ be the set of all functions $G\dvtx  R\to\RR$ such
that $V_p(G;R)$ is bounded,
which extends Definition~\ref{Vitali} when $p=1$.

The following is an elaboration on the statements 3.7(ii) and 4.3 of
\cite{LandL1984}.
In the present case, we do not assume $f(a)=0$ and $g(c)=0$.

\begin{teo}\label{LL1}
Let $R=[a,b]\times[c,d]$ for some real numbers $a<b$ and $c<d$.
Let $p>1$ and $q>1$ be such that $p^{-1}+q^{-1}>1$.
Let $f\in\CLW_q[a,b]$, $g\in\CLW_q[c,d]$ and\vspace*{1pt} let $G\in\CLW_{p}(R)$
be continuous.
There exists the double Riemann--Stieltjes integral $\int
_a^b\int_c^df \otimes g \,\mathrm{d}^2G$
and\vspace*{-3pt}
%
\begin{equation}
\label{1LL1} \biggl|\int_a^b\int
_c^d\bigl[f \otimes g-f(a)g(c)\bigr]
\,\mathrm{d}^2G \biggr|\leq 8K_{p,q} \bigl[\|f\|_{[q]}V_q(g)+
\|g\|_{[q]}V_q(f) \bigr]V_{p}(G;R),
\end{equation}
where $K_{p,q}:=(1+\zeta(p^{-1}+q^{-1} ))^2$ and
$\zeta(s):=\sum_{k=1}^{\infty}k^{-s}$ for $s>1$.
\end{teo}

\begin{pf}
The functions $\Phi_p$ and $\Phi_q$ with
values $\Phi_p(x):=x^p/p$ and $\Phi_q(x):=x^q/q$ for $x\geq0$ are the
$N$-functions.
We apply the results of Le\'sniewicz and Le\'sniewicz \cite{LandL1984} for
$\Phi=\Psi=\Phi_p$ and $\tilde\Phi=\tilde\Psi=\Phi_q$.
The integral $\int_a^b\int_c^df \otimes g \,\mathrm{d}^2G$ exists
by Theorem~4.3 in  \cite{LandL1984}, page 57.
(We note that continuity of the functions $f$ and $g$ is not used in
the proof there.)
To obtain the bound (\ref{1LL1}), it is enough to bound the
Riemann--Stieltjes sums.
Let $\dot{\tau}=\{(u_i,v_j),[s_{i-1},s_i]\times[t_{j-1},t_j]\dvtx
(i,j)\in(n]\times(m]\}$
be a tagged partition of $R$.
Letting $u_0:=a$ and $v_0:=c$ we have the identity\vspace*{-3pt}
\begin{eqnarray*}
S\bigl(f\otimes g-f(a)g(c),\Delta^2G;\dot{\tau}\bigr)&=&\sum
_{i=1}^n\sum_{j=1}^m
\sum_{k=1}^i\sum
_{l=1}^j \Delta_kf
\Delta_lg\Delta_{i,j}^{\tau}G+f(a)\sum_{j=1}^m\sum
_{l=1}^j\Delta_lg\sum
_{i=1}^n\Delta _{i,j}^{\tau}G
\\[-2pt]
&&{}
+g(c)\sum_{i=1}^n\sum
_{k=1}^i\Delta_kf\sum
_{j=1}^m\Delta _{i,j}^{\tau}G,
\end{eqnarray*}
where $\Delta_kf:=f(u_k)-f(u_{k-1})$ and $\Delta_lg:=g(v_l)-g(v_{l-1})$.
Using the bounds 3.5 in \cite{LandL1984}, page
53, and (5.1) in  \cite{LCY1936}, page
254,
we get\vspace*{-3pt}
\begin{eqnarray*}
&&\bigl|S\bigl(f\otimes g-f(a)g(c),\Delta^2G;\dot{\tau}\bigr) \bigr|
\\[-2pt]
&&\quad \leq 16 \biggl(1+\zeta \biggl(\frac{1}{p}+\frac{1}{q} \biggr)
\biggr)^2 V_q\bigl(f;[a,b]\bigr)V_q
\bigl(g;[a,b]\bigr)V_p(G;R)
\\[-2pt]
&&\qquad {} + \biggl(1+\zeta \biggl(\frac{1}{p}+\frac{1}{q} \biggr) \biggr)
\bigl[\bigl|f(a)\bigr|V_q\bigl(g;[c,d]\bigr)+\bigl|g(c)\bigr|V_q
\bigl(f;[a,b]\bigr) \bigr]V_p(G;R),
\end{eqnarray*}
and so (\ref{1LL1}) follows.
Instead of the bound 3.3 in  \cite{LandL1984},
page 51,
we used the bound of L.C. Young since it gives a smaller constant in
(\ref{1LL1})
in the present setting.
\end{pf}

We use the following two versions of the preceding inequality (\ref{1LL1})
adapted to subrectangles of a rectangle $[0,T]^2$.

\begin{cor}\label{LL2}
Let $p>1$ and $q>1$ be such that $p^{-1}+q^{-1}>1$.
Let $T>0$, let $f\in\CLW_q[0,T]$ and let $G\in\CLW_{p}[0,T]^2$ be
continuous.
There exists the double Riemann--Stieltjes integral \mbox{$\int
_0^T\int_0^Tf \otimes f \,\mathrm{d}^2G$}.
Also, letting $K_{p,q}:=16(1+\zeta(p^{-1}+q^{-1} ))^2$,
\begin{enumerate}[(ii)]
\item[(i)] the inequality
%
\begin{equation}
\label{1LL2} \biggl|\int_s^t\int
_s^t \bigl[f\otimes f-f^2(s)
\bigr] \,\mathrm{d}^2G \biggr| \leq K_{p,q}\|f\|_{[q],[0,T]}V_q
\bigl(f;[s,t]\bigr)V_p\bigl(G;[s,t]^2\bigr)
\end{equation}
holds for any $0\leq s<t\leq T$;
\item[(ii)] the inequality
%
\begin{equation}
\label{2LL2} \biggl|\int_s^t\int
_u^vf\otimes f \,\mathrm{d}^2G \biggr|\leq
K_{p,q}\|f\|_{[q],[0,T]}^2 V_p
\bigl(G;[s,t]\times[u,v]\bigr)
\end{equation}
holds for any $0\leq s<t\leq T$ and $0\leq u<v\leq T$.
\end{enumerate}
\end{cor}

\subsection*{The quadratic mean Riemann--Stieltjes integral}
This integral is defined for a (deterministic) function
with respect to a stochastic process in the present paper.
Let $X=\{X(t)\dvtx  t\geq0\}$ be a second order stochastic
process on a probability space $(\Omega,\mathcal{F},\Pr)$, which is
a family of random variables $X(t)$ having mean zero and finite second moment.
The covariance function of $X$ is the function $\Gamma_X$ defined on
$\RR_{+}^2=[0,\infty)\times[0,\infty)$ with values
\[
\Gamma_X(s,t):=E\bigl[X(s)X(t)\bigr],\qquad  (s,t)\in\RR_{+}^2.
\]
Let $f\dvtx [0,\infty)\to\RR$ be a function and let $0\leq
a<b<\infty$.
For a tagged partition $\dot{\kappa}=\{(u_i,[t_{i-1},t_i])\dvtx  i\in
(n]\}$ of
the interval $[a,b]$ the Riemann--Stieltjes sum is
\[
S_{\mathrm{RS}}(f,\Delta X;\dot{\kappa}):=\sum_{i=1}^nf(u_i)
\bigl[X(t_i)-X(t_{i-1})\bigr],
\]
and so it is a random variable in $L^2(\Omega,\mathcal{F},\Pr)$.
We say that the \emph{quadratic mean Riemann--Stieltjes} integral over $[a,b]$
of $f$ with respect to $X$ exists and equals
$I\in L^2(\Omega,\mathcal{F},\Pr)$, if for each $\epsilon>0$ there
is a $\delta>0$ such that
\[
E \bigl[S_{\mathrm{RS}}(f,\Delta X;\dot{\kappa})-I \bigr]^2<
\epsilon
\]
for each tagged partition $\dot{\kappa}$ of $[a,b]$ with the mesh
$|\kappa|<\delta$.
If such $I$ exists, then it is unique in $L^2$ and is denoted by
\[
\int_a^bf \,\mathrm{d}X =\mathrm{q.m.}\int_a^b
f(t) \,\mathrm{d}X(t):=I.
\]
Next, is the integration in quadratic mean criterion of Lo\`eve \cite{ML1978}, page 138.

%
\begin{prop}\label{Loeve1}
Let $X$ be a second order stochastic process and
$f\dvtx [0,\infty)\to\RR$.
For $0\leq a<b<\infty$, the quadratic mean Riemann--Stieltjes integral
\[
\int_a^bf \,\mathrm{d}X  \mbox{ exists}\quad  \mbox{if and only if}\quad
\int_a^b\int_a^bf
\otimes f \,\mathrm{d}^2\Gamma_X \mbox{ exists}
\]
as the double Riemann--Stieltjes integral.
Moreover, for any $0\leq s<t<\infty$ and $0\leq u<v<\infty$ if
the two integrals $\int_s^tf \,\mathrm{d}X$ and $\int_u^vf \,\mathrm{d}X$
exist then so does $\int_s^t\int_u^vf\otimes f \,\mathrm{d}^2\Gamma_X$
and the equality
%
\begin{equation}
\label{1Loeve1} E \biggl[\int_s^tf \,\mathrm{d}X\int
_u^vf \,\mathrm{d}X \biggr]=\int_s^t
\int_u^vf \otimes f \,\mathrm{d}^2
\Gamma_X
\end{equation}
holds.
\end{prop}

Formal properties of Riemann--Stieltjes integrals such as (finite) additivity
and linearity hold almost surely for corresponding integrals in
quadratic mean.

We shall write $\CLQ_p:=[1,p/(p-1))$ if $p>1$ and $\CLQ_1:=\{\infty\}$.
The following theorem holds by Theorem~\ref{DRS}, Corollary~\ref{LL2}
and Proposition~\ref{Loeve1}.

\begin{teo}\label{integral}
Let $X$ be a second order stochastic process with the continuous
covariance function
$\Gamma_X\in\CLW_p[0,T]^2$ for some $p\geq1$ and $0<T<\infty$, and let
$f\in\CLW_q[0,T]$ with $q\in\CLQ_p$.
Then for each $t\in[0,T]$ there exists the q.m. Riemann--Stieltjes integral
$\int_0^tf \,\mathrm{d}X$ and there is a finite constant $K=K(p,f)$ (depending on $p$
and $f$) such that the inequality
\[
E \biggl[\mathrm{q.m.}\int_s^t f \,\mathrm{d}X
\biggr]^2\leq KV_p\bigl(\Gamma_X;[s,t]^2
\bigr)
\]
holds for any $0\leq s<t\leq T$.
\end{teo}

Given a second order stochastic process $X$, a class of functions $f$
such that
$\int_0^Tf \,\mathrm{d}X$ is defined as the quadratic mean Riemann--Stieltjes integral
can be larger than the class of functions $f$ such that $\int_0^T
f \,\mathrm{d}X$ is
defined as the pathwise Riemann--Stieltjes integral.
Indeed, let $X$ be a fractional Brownian motion $B_H$ with the Hurst index
$H\in(0,1)$.
By Proposition~\ref{fBm1} below, $B_H$ has the continuous covariance
function $\Gamma_{B_H}\in\CLW_p[0,T]^2$ with $p=\max\{1,1/(2H)\}$.
Therefore, the q.m. Riemann--Stieltjes integral $\int_0^Tf \,\mathrm{d}B_H$
is defined
for each $f\in\CLW_q[0,T]$, where
\[
q<\frac{1}{1-2H} \qquad \mbox{if } H\in(0,1/2)\quad  \mbox{and}\quad  q=\infty\qquad \mbox{if } H
\in[1/2,1)
\]
by the preceding theorem.
While the pathwise Riemann--Stieltjes integral $\int_0^Tf \,\mathrm{d}B_H$ is defined
for each $f\in\CLW_q[0,T]$ with $q<1/(1-H)$ if $H\in(0,1)$
by the result of Young \cite{LCY1936},
and these results are best possible in terms of $p$-variation
(see \cite{DandN2010}, Section~3.7).

The preceding comment suggests that a family of random
variables
%
\begin{equation}
\label{qmRS} \mathrm{q.m.}\int_0^tf \,\mathrm{d}X, \qquad t\in[0,T],
\end{equation}
need not be a stochastic process with well-behaved sample functions.
The following is a standard approach to deal with such cases.

\begin{teo}\label{Y2}
Suppose that the hypotheses of the preceding theorem hold.
Suppose that for each $t\in(0,T]$
\[
\lim_{s\uparrow t}V_p\bigl(\Gamma_X;[s,t]^2
\bigr)=0.
\]
Then a measurable and separable stochastic process $Y=\{Y(t)\dvtx  t\in
[0,T]\}$
exists on $(\Omega,\mathcal{F},\Pr)$ such that
\[
\Pr \biggl( \biggl\{Y(t)=\mathrm{q.m.}\int_0^tf \,\mathrm{d}X
\biggr\} \biggr)=1
\]
for each $t\in[0,T]$.
\end{teo}

Throughout the paper, we assume that the q.m. Riemann--Stieltjes
integrals (\ref{qmRS})
are given by the stochastic process $Y$ from the preceding theorem,
to be called the \emph{q.m. integral process}.


\section{$p$-variation of the covariance function}\label{covariance}


We start with a simple fact concerning the boundedness of variation of
the covariance functions
of stochastic processes with positively or negatively correlated
disjoint increments
(meaning hypothesis (\ref{1SP1}) or (\ref{1SP2}), respectively).

\begin{prop}\label{SP1}
Let $0<T<\infty$ and let $X=\{X(t)\dvtx  t\in[0,T]\}$ be a second
order stochastic process
with the covariance function $\Gamma_X$ on $[0,T]^2$.
\begin{enumerate}[(ii)]
\item[(i)] If for any $0\leq u<v\leq s<t\leq T$,
%
\begin{equation}
\label{1SP1} E\bigl[X(v)-X(u)\bigr] \bigl[X(t)-X(s)\bigr]\geq0,
\end{equation}
then for any $0\leq a< b\leq T$ and $0\leq c<d\leq T$
\[
V_1\bigl(\Gamma_X;[a,b]\times[c,d]\bigr)=E
\bigl[X(b)-X(a)\bigr] \bigl[X(d)-X(c)\bigr].
\]
\item[(ii)] If for any $0\leq u<v\leq s<t\leq T$,
%
\begin{equation}
\label{1SP2} E\bigl[X(v)-X(u)\bigr] \bigl[X(t)-X(s)\bigr]\leq0,
\end{equation}
then for any $0\leq a< b\leq c<d\leq T$
%
\begin{equation}
\label{2SP2} V_1\bigl(\Gamma_X;[a,b]\times[c,d]
\bigr)= \bigl|E\bigl[X(b)-X(a)\bigr] \bigl[X(d)-X(c)\bigr] \bigr|.
\end{equation}
\end{enumerate}
\end{prop}

\begin{pf}
To prove (i) note that (\ref{1SP1}) holds for any pairs of
closed intervals $[u,v]$ and $[s,t]$ in $[0,T]$ provided (\ref{1SP1})
holds for such intervals having at most a common endpoint, as assumed.
Then the conclusion follows using the relation
\[
\bigl|\Delta^{[u,v]\times[s,t]}\Gamma_X\bigr|=E\bigl[X(v)-X(u)\bigr]
\bigl[X(t)-X(s)\bigr].
\]
In the case (ii), the conclusion follows using the relation
\[
\bigl|\Delta^{[u,v]\times[s,t]}\Gamma_X\bigr|=-E\bigl[X(v)-X(u)\bigr]
\bigl[X(t)-X(s)\bigr]
\]
for nonoverlapping intervals $[u,v]$ and $[s,t]$ in $[0,T]$.
\end{pf}

By the second part of the preceding proposition the covariance function of
a stochastic process with negatively correlated disjoint increments has
bounded variation over rectangles which do not contain a diagonal.
The following result for such a process, with an additional assumption
(\ref{3SP2}), gives a bound of the $p$-variation of the covariance function
over rectangles containing a diagonal.

\begin{teo}\label{SP2}
Let $0<T<\infty$, let $p\geq1$ and let $X=\{X(t)\dvtx  t\geq0\}$ be
a second order stochastic process
with the covariance function $\Gamma_X$ such that \textup{(\ref
{1SP2})} holds
for any $0\leq u<v\leq s<t\leq T$, and
%
\begin{equation}
\label{3SP2} E\bigl[X(v)-X(u)\bigr] \bigl[X(t)-X(s)\bigr]\geq0,
\end{equation}
holds for any $0\leq s\leq u<v\leq t\leq T$.
Then for any $0\leq a<b\leq T$
%
\begin{equation}
\label{4SP2} V_p\bigl(\Gamma_X;[a,b]^2
\bigr)\leq2V_{2p}\bigl(\psi_X;[a,b]\bigr)^{2},
\end{equation}
where $\psi_X\dvtx [0,T]\to L^2(\Omega,\mathcal{F},\Pr)$ defined by
$\psi_X(t):=X(t,\cdot)$ for $t\in[0,T]$.
\end{teo}

\begin{rem}\label{SP3}
The theorem is meaningful provided the right side of (\ref{4SP2})
is finite.
In addition to the hypotheses of Theorem~\ref{SP2},
suppose that $X$ and $p\geq1$ are such that for a finite constant $L$
the inequality
\[
E\bigl[X(t)-X(s)\bigr]^2\leq L(t-s)^{\sfrac{1}{p}}
\]
holds for each $0\leq s<t\leq T$.
Then for any $0\leq a<b\leq T$ we have
$V_{2p}(\psi_X;[a,b])\leq\linebreak[4] \sqrt{L}(b-a)^{1/(2p)}$,
and so by Theorem~\ref{SP2}
\[
V_p\bigl(\Gamma_X;[a,b]^2\bigr)
\leq2L(b-a)^{\sfrac{1}{p}}.
\]
\end{rem}

\begin{pf*}{Proof of Theorem~\ref{SP2}}
Let $0\leq a<b\leq T$.
Without loss of generality, we can assume that the right side of
(\ref{4SP2}) is finite.
Let $\lambda\times\kappa=\{(s_i,t_j)\dvtx
i\in[n], j\in[m]\}$ be a partition of $[a,b]^2$
with $n\geq1$ and $m\geq1$.
If $n=1$ then, since $p\geq1$ and (\ref{3SP2}) holds, we have
%
\begin{eqnarray}\label{5SP2}
s_p(\Gamma_X;\lambda\times\kappa)&\leq& \Biggl(\sum
_{j=1}^mE\bigl[X(b)-X(a)\bigr]
\Delta_j^{\kappa}X \Biggr)^p= \bigl(E
\bigl[X(b)-X(a)\bigr]^2 \bigr)^p
\nonumber
\\[-8pt]\\[-8pt]
&=&\bigl\|\psi_X(b)-\psi_X(a)\bigr\|_{L_2}^{2p}
\leq V_{2p}\bigl(\psi _X;[a,b]\bigr)^{2p}.\nonumber
\end{eqnarray}
Let $n\geq2$, let $1\leq i\leq n$, and let $A_i:=\{j\in(m-1]\dvtx
t_j\in(s_{i-1},s_i)\}$.
If $A_i$ is the empty set, then there is a $j_0\in(m]$ such that
$[s_{i-1},s_i]\subset[t_{j_0-1},t_{j_0}]$.
In this case, we have
\begin{eqnarray*}
\bigl|E\Delta_i^{\lambda}X\Delta_{j_0}^{\kappa}X\bigr|
&\leq& E\bigl[\Delta _i^{\lambda}X\bigr]^2
\\
&&{}+ \bigl|E\Delta_i^{\lambda}X \bigl[X(t_{j_0})-X(s_i)
\bigr] \bigr| + \bigl|E\Delta_i^{\lambda}X \bigl[X(s_{i-1})-X(t_{j_0-1})
\bigr] \bigr|.
\end{eqnarray*}
If $A_i$ is not the empty set then let $j_1$ be the minimal element in $A_i$
and let $j_2$ be the maximal element in $A_i$.
In this case, we have
\[
\bigl|E\Delta_i^{\lambda}X\Delta_{j_1}^{\kappa}X\bigr|
\leq \bigl|E\Delta_i^{\lambda}X\bigl[X(t_{j_1})-X(s_{i-1})
\bigr] \bigr| + \bigl|E\Delta_i^{\lambda}X\bigl[X(s_{i-1})-X(t_{j_1-1})
\bigr] \bigr|
\]
and
\[
\bigl|E\Delta_i^{\lambda}X\Delta_{j_2+1}^{\kappa}X\bigr|
\leq \bigl|E\Delta_i^{\lambda}X\bigl[X(t_{j_2+1})-X(s_{i})
\bigr] \bigr| + \bigl|E\Delta_i^{\lambda}X\bigl[X(s_{i})-X(t_{j_2})
\bigr] \bigr|.
\]
Therefore to bound $\sum_{j=1}^m|E\Delta_i^{\lambda}X\Delta
_j^{\kappa}X|$,
we can and do assume that in the partition $\kappa$ we\vspace*{1pt} have
$t_{j_1}=s_{i-1}$ and $t_{j_2}=s_i$ for some $j_1<j_2$ in $(m-1]$.
Using this assumption and negative correlation for disjoint increments
it follows that
\[
\sum_{j=1}^m\bigl|E\Delta_i^{\lambda}X
\Delta_j^{\kappa}X\bigr| =2E\bigl[\Delta_i^{\lambda}X
\bigr]^2-E\Delta_i^{\lambda}X\bigl[X(b)-X(a)\bigr]
\leq2E\bigl[\Delta_i^{\lambda}X\bigr]^2,
\]
where the last inequality holds by (\ref{3SP2}).
Finally, since $p\geq1$, we have
\begin{eqnarray*}
s_p(\Gamma_X;\lambda\times\kappa)&\leq&\sum
_{i=1}^n \Biggl(\sum_{j=1}^m
\bigl|E\Delta_i^{\lambda}X\Delta_j^{\kappa}X\bigr|
\Biggr)^p \leq2^p\sum_{i=1}^n
\bigl(E\bigl[\Delta_i^{\lambda}X\bigr]^2
\bigr)^p
\\
&=&2^p\sum_{i=1}^n\bigl\|
\psi_X(s_i)-\psi_X(s_{i-1})
\bigr\|_{L_2}^{2p} \leq2^pV_{2p}\bigl(
\psi_X;[a,b]\bigr)^{2p}.
\end{eqnarray*}
Recalling the bound (\ref{5SP2}) in the case $n=1$, the conclusion
(\ref{4SP2}) follows.
\end{pf*}

Next, we show that for several classes of stochastic processes including
fractional Brownian motion, subfractional Brownian motion and bifractional
Brownian motion
one has positively or negatively correlated increments.

\subsection*{Stochastic processes with stationary increments}
First, consider real-valued stochastic processes $X$ with mean zero,
finite second
moments $E[X(t)]^2$ and (weakly) stationary increments.
Then the incremental variance function $\sigma_X^2(t,\allowbreak t+r)$ does
not depend on $t$, and so it is a function of $r$.
The following fact is known (see \cite{MBM1968},
page 32); we sketch a
proof for
completeness.

\begin{lem}\label{SP4}
Let $X=\{X(t)\dvtx  t\geq0\}$ be a mean zero second order stochastic process
with stationary increments, and let $\phi\dvtx [0,\infty)\to
[0,\infty)$ be
the function with values
\[
\phi(r):=\sigma_X^2(t,t+r)=E\bigl[X(t+r)-X(t)
\bigr]^2
\]
for each $r\geq0$.
\begin{enumerate}[(ii)]
\item[(i)] If $\phi$ is convex on $[0,T]$, then \textup{(\ref
{1SP1})} holds
for any $0\leq u<v\leq s<t\leq T$.
\item[(ii)] If $\phi$ is concave on $[0,T]$, then \textup{(\ref
{1SP2})} holds
for any $0\leq u<v\leq s<t\leq T$.
\end{enumerate}
\end{lem}

\begin{pf}
To prove (i) let $\phi$ be convex on $[0,T]$, and let
$0\leq u<v\leq s<t\leq T$.
Using an expression of $\phi$ in terms of the covariance function
$\Gamma_X$,
it follows that
\[
2E\bigl[X(v)-X(u)\bigr] \bigl[X(t)-X(s)\bigr]=\bigl[\phi(t-u)-\phi(t-v)\bigr]-
\bigl[\phi(s-u)-\phi(s-v)\bigr].
\]
Inserting additional points in the interval $[u,v]$ if necessary,
one can suppose that $v-u<t-s$.
Then letting $x_1:=s-v$, $x_2:=s-u$, $x_3:=t-v$ and $x_4:=t-u$,
we have $0\leq x_1<x_2<x_3<x_4\leq T$ and
\[
\frac{\phi(x_4)-\phi(x_3)}{x_4-x_3}\geq\frac{\phi(x_3)-\phi
(x_2)}{x_3-x_2} \geq\frac{\phi(x_2)-\phi(x_1)}{x_2-x_1},
\]
by convexity of $\phi$.
This proves (\ref{1SP1}), and so (i).
The proof of (ii) is symmetric.
\end{pf}

We apply this fact to a fractional Brownian motion $B_H=\{B_H(t)\dvtx
t\in[0,T]\}$
with the Hurst index $H\in(0,1)$,
which is a Gaussian stochastic process with mean zero and the
covariance function
\[
F_H(s,t):=\Gamma_{B_H}(s,t)=\tfrac{1}{2}
\bigl(t^{2H}+s^{2H}-|t-s|^{2H} \bigr)
\]
for $(s,t)\in[0,T]^2$.

\begin{prop}\label{fBm1}
Let $B_H$ be a fractional Brownian motion with the Hurst index $H\in(0,1)$,
let $\rho_H(u):=u^H$ for each $u\in[0,T]$, and let $p:=\max\{
1,1/(2H)\}$.
Then $B_H\in\mathcal{LSI}(\rho_H(\cdot))$.
Also, the inequality
%
\begin{equation}
\label{1fBm1} V_p\bigl(F_H;[s,t]^2\bigr)
\leq C_1\bigl[\rho_H(t-s)\bigr]^2\equiv
C_1(t-s)^{2H}
\end{equation}
with $C_1=1$ if $2H\geq1$ and $C_1=2$ if $2H<1$,
holds for any $0\leq s<t\leq T$,
and the inequality
%
\begin{equation}
\label{2fBm1} \sum_{j=1}^mV_p
\bigl(F_H;J_i^{\kappa}\times J_j^{\kappa}
\bigr) \leq\sum_{j=1}^m\bigl|E\bigl[
\Delta_i^{\kappa}B_H\Delta_j^{\kappa}B_H
\bigr]\bigr| \leq C_2\bigl(\Delta_i^{\kappa}
\bigr)^{1\wedge(2H)}
\end{equation}
with $C_2=2HT^{2H-1}$ if $2H\geq1$ and $C_2=3$ if $2H<1$,
holds for any partition $\kappa=(t_j)_{j\in[m]}$
of $[0,T]$ and for any $i\in(m]$.
\end{prop}

\begin{pf}
The incremental variance function $\sigma_{B_H}^2(s,t)=|t-s|^{2H}$ for
$(s,t)\in[0,T]$.
Clearly, $B_H\in\mathcal{LSI}(\rho_H(\cdot))$.
By Lemma~\ref{SP4}, disjoint increments of $B_H$ are positively
correlated if
$2H\geq1$ and negatively correlated if $2H< 1$.
If $2H\geq1$, then by part (i) of Proposition~\ref{SP1},
\[
V_1\bigl(F_H;[a,b]^2\bigr)=E
\bigl[B_H(b)-B_H(a)\bigr]^2=(b-a)^{2H}
\]
for any $0\leq a<b\leq T$, proving (\ref{1fBm1}) with $C_1=1$ in this case.
If $2H< 1$ then (\ref{1fBm1}) holds with $C_1=2$ by Remark~\ref{SP3} and
Theorem~\ref{SP2} since its hypothesis (\ref{3SP2})
holds due to the relation
%
\begin{eqnarray}\label{3fBm1}
&&2E\bigl[B_H(v)-B_H(u)\bigr] \bigl[B_H(t)-B_H(s)
\bigr]
\nonumber
\\[-8pt]\\[-8pt]
&&\quad =(v-s)^{2H}-(u-s)^{2H}+(t-u)^{2H}-(t-v)^{2H}>0\nonumber
\end{eqnarray}
for any $0\leq s\leq u<v\leq t\leq T$.

To prove (\ref{2fBm1}) let $\kappa=(t_j)_{j\in[m]}$ be a partition
of $[0,T]$, and let
$i\in(m]$.
Due to (\ref{1fBm1}) one can suppose that $m>1$.
First let $H\in[1/2,1)$.
Then $p=1$.
By part (i) of Lemma~\ref{SP4} and by part (i) of Proposition~\ref{SP1}
we have
\begin{eqnarray*}
\sum_{j=1}^mV_1
\bigl(F_H;[t_{i-1},t_i]\times[t_{j-1},t_j]
\bigr) &=&\sum_{j=1}^mE\bigl[
\Delta_i^{\kappa}B_H\Delta_j^{\kappa}B_H
\bigr]
\\
&=&E\bigl[B_H(t_i)-B_H(t_{i-1})
\bigr] \bigl[B_H(T)-B_H(0)\bigr]
\\
&=&\frac{1}{2} \bigl[t_i^{2H}-t_{i-1}^{2H}+(T-t_{i-1})^{2H}-(T-t_i)^{2H}
\bigr]
\\
&\leq& 2HT^{2H-1}(t_i-t_{i-1}),
\end{eqnarray*}
where the last inequality holds by the mean value theorem.
Now, let $H\in(0,1/2)$.
Then $p=1/(2H)>1$.
Since $V_p(\cdot)\leq V_1(\cdot)$,
by part (ii) of Lemma~\ref{SP4} and by part (ii) of
Proposition~\ref{SP1} we have
\begin{eqnarray*}
\sum_{j\in(m]\setminus\{i\}}V_p\bigl(F_H;J_i^{\kappa}
\times J_j^{\kappa}\bigr) &\leq&\sum
_{j\in(m]\setminus\{i\}}V_1\bigl(F_H;J_i^{\kappa}
\times J_j^{\kappa}\bigr)
\\
&=&\sum_{j\in(m]\setminus\{i\}}\bigl|E\bigl(\Delta_i^{\kappa}B_H
\Delta _j^{\kappa}B_H\bigr)\bigr|\\
& =&E\bigl(
\Delta_i^{\kappa}B_H\bigr)^2-E
\Delta_i^{\kappa}B_H\bigl[B_H(T)-B_H(0)
\bigr]
\\
&\leq& E\bigl(\Delta_i^{\kappa}B_H
\bigr)^2=(t_i-t_{i-1})^{2H},
\end{eqnarray*}
where the last inequality holds by (\ref{3fBm1}).
This together with (\ref{1fBm1}) gives (\ref{2fBm1}),
completing the proof.
\end{pf}

The inequality (\ref{1fBm1}) in the case $H\in(0,1/2)$ and with a
different constant
is the same as the one stated by Proposition~13 in \cite{FandV2010}.
The proofs seem to be also different.

\subsection*{Sub-fractional Brownian motion}
Let $H\in(0,1)$ and $0<T<\infty$.
The function $R_H\dvtx [0,T]^2\to\RR$ with values
%
\begin{equation}
\label{sfBm1} R_H(s,t):=s^{2H}+t^{2H}-
\tfrac{1}{2} \bigl[(s+t)^{2H}+|s-t|^{2H} \bigr],
\end{equation}
$(s,t)\in[0,T]^2$, is positive definite as shown in  \cite{BGT04}.
A \emph{sub-fractional Brownian motion} with index $H$ is a mean zero
Gaussian stochastic process $G_H=\{G_H(t)\dvtx  t\in[0,T]\}$ with the
covariance function $R_H$ and with the incremental variance function
\[
\sigma_{G_H}^2(s,t)=|s-t|^{2H}+(s+t)^{2H}-2^{2H-1}
\bigl[t^{2H}+s^{2H}\bigr]
\]
for $s,t\in[0,T]$.
In the case $H=1/2$, $G_{1/2}$ is a Brownian motion.
A subfractional Brownian motion $G_H$ with index $H$ is $H$-self-similar
but does not have stationary increments if $H\neq1/2$.

\begin{prop}\label{sfBm2}
Let $G_H=\{G_H(t)\dvtx  t\in[0,T]\}$
be a sub-fractional Brownian motion with $H\in(0,1)$.
The following properties hold:
\begin{enumerate}[(iii)]
\item[(i)] for any $0\leq s\leq t\leq T$
\begin{eqnarray*}
(t-s)^{2H}&\leq&\sigma_{G_H}^2(s,t)\leq
\bigl(2-2^{2H-1}\bigr) (t-s)^{2H}, \qquad \mbox{if } 0<H<1/2,
\\
\bigl(2-2^{2H-1}\bigr) (t-s)^{2H}&\leq&\sigma_{G_H}^2(s,t)
\leq (t-s)^{2H}, \qquad \mbox{if } 1/2<H<1;
\end{eqnarray*}
\item[(ii)] for any $0\leq u<v\leq s<t\leq T$
\[
E\bigl[G_H(v)-G_H(u)\bigr] \bigl[G_H(t)-G_H(s)
\bigr] \lleft\{ %
\begin{array} {l@{\qquad}l} < 0, &\mbox{if $0<H<1/2$,}
\\
> 0, &\mbox{if $1/2<H<1$};
\end{array} %
\rright.
\]
\item[(iii)] for any $0\leq s\leq u<v\leq t\leq T$
\[
C(u,v,s,t):=E\bigl[G_H(v)-G_H(u)\bigr]
\bigl[G_H(t)-G_H(s)\bigr]> 0.
\]
\end{enumerate}
\end{prop}

\begin{pf}
Statements (i) and (ii) are proved in  \cite
{BGT04}, Theorems (3), (5).
To prove (iii), let $0\leq s\leq u<v\leq t\leq T$.
Since the pairs of intervals $[s,u]$, $[u,v]$ and $[v,t]$, $[u,v]$ do
not intersect
(except for the endpoints),
by (ii) in the case $1/2<H<1$, we have
\begin{eqnarray*}
C(u,v,s,t)&=&E\bigl[G_H(v)-G_H(u)\bigr]^2
\\
&&{}+E\bigl[G_H(v)-G_H(u)\bigr] \bigl[G_H(u)-G_H(s)
\bigr]\\
&&{}+E\bigl[G_H(v)-G_H(u)\bigr] \bigl[G_H(t)-G_H(v)
\bigr]\\
&>& 0.
\end{eqnarray*}
Thus we can suppose that $0<H<1/2$.
Using the values of the covariance function (\ref{sfBm1}) it follows that
%
\begin{eqnarray}\label{1sfBm2}
C(u,v,s,t)&=&\tfrac{1}{2} \bigl\{-(t+v)^{2H}-(t-v)^{2H}+(t+u)^{2H}+
(t-u)^{2H}\nonumber
\\[-8pt]\\[-8pt]
&&\hphantom{\tfrac{1}{2} \bigl\{}{} +(v+s)^{2H}+(v-s)^{2H}-(u+s)^{2H}-(u-s)^{2H}
\bigr\}.
\nonumber
\end{eqnarray}
Let $f_t(x):=(t+x)^{2H}+(t-x)^{2H}$ for each $x\in[0,t]$.
Since $0<H<1/2$, then
$f_t'(x)<0$ for $x\in(0,t)$, and so for $0\leq u<v\leq t$ we have
$-f_t(v)+f_t(u)> 0$
by the mean value theorem.
Let $g_s(x):=(x+s)^{2H}+(x-s)^{2H}$ for each $x\geq s$.
Since $g_s'(x)>0$ for $x>s$ and $v>u\geq s$, we have $g_s(v)-g_s(u)>0$
by the mean value theorem again.
Therefore
\[
C(u,v,s,t)=\tfrac{1}{2} \bigl\{-f_t(v)+f_t(u)+g_s(v)-g_s(u)
\bigr\}>0,
\]
as claimed.
\end{pf}

The following proposition shows that the hypotheses (\ref{G7}) and
(\ref{3asconvergence}) of the main result (Theorem~\ref{asconvergence})
hold true for a sub-fractional Brownian motion.

\begin{prop}\label{sfBm3}
Let $G_H=\{G_H(t)\dvtx  t\in[0,T]\}$ be a sub-fractional Brownian
motion with $H\in(0,1)$,
let $\rho_H(u):=u^H$ for each $u\in[0,T]$, and let $p:=\max\{
1,1/(2H)\}$.
Then $G_H\in\mathcal{LSI}(\rho_H(\cdot))$.
Also, there is a finite constant $C_1$ such that the inequality
%
\begin{equation}
\label{1sfBm3} V_p\bigl(R_H;[s,t]^2\bigr)
\leq C_1\bigl[\rho_H(t-s)\bigr]^2\equiv
C_1(t-s)^{2H}
\end{equation}
holds for any $0\leq s<t\leq T$, and there is a finite constant $C_2$
such that the inequality
%
\begin{equation}
\label{2sfBm3} \sum_{j=1}^mV_p
\bigl(R_H;J_i^{\kappa}\times J_j^{\kappa}
\bigr) \leq\sum_{j=1}^m \bigl|E\bigl[
\Delta_i^{\kappa}G_H\Delta_j^{\kappa
}G_H
\bigr] \bigr| \leq C_2\bigl(\Delta_i^{\kappa}
\bigr)^{1\wedge(2H)}
\end{equation}
holds for any partition $(t_i)_{i\in[m]}$ of $[0,T]$ and for any $i\in(m]$.
\end{prop}

\begin{pf}
Condition (A1) of Definition~\ref{LV} holds by part (i) of
Proposition~\ref{sfBm2}.
To prove condition (A2) suppose that $H\in(0,1/2)\cup(1/2,1)$ and
let $\epsilon\in(0,T)$.
For each $s\in[\epsilon,T)$ and \mbox{$t\in[-\epsilon,T-s]$}, let
\[
f_s(t):=(2s+t)^{2H}-2^{2H-1}
\bigl[s^{2H}+(s+t)^{2H} \bigr].
\]
Then $f_s(0)=f_s'(0)=0$ and
\[
b(s,s+h):= \biggl(\frac{\sigma_{G_H}(s,s+h)}{\rho_H(h)} \biggr)^2-1
=h^{-2H}f_s(h).
\]
Let $s\in[\epsilon,T)$ and $h\in(0,T-s]$.
By Taylor's theorem with the Lagrange remainder applied to the function
$f_s$, there exists $u=u(s,h)\in(0,h)$ such that
\[
b(s,s+h)=2^{-1}h^{2(1-H)}f_s''(u),
\]
where
\[
f_s''(u)=2H(2H-1)
\bigl[(2s+u)^{2H-2}-2^{2H-1}(s+u)^{2H-2} \bigr].
\]
Then there is a finite constant $C=C(\epsilon,H)$ such that the inequality
\[
\bigl|b(s,s+h)\bigr|\leq Ch^{2(1-H)}
\]
holds for each $s\in[\epsilon,T)$ and $h\in(0,T-s]$.
Since $H<1$ the preceding bound yields that condition (A2) holds,
and so $G_H\in\mathcal{LSI}(\rho_H(\cdot))$.

To prove (\ref{1sfBm3}), first let $H\in(1/2,1)$.
Then $p=1$ and hypothesis (\ref{1SP1}) holds for $X=G_H$ by part
(ii) of Proposition~\ref{sfBm2}.
Therefore in this case
by part (i) of Proposition~\ref{SP1} and by part (i) of Proposition~\ref{sfBm2}, for any $0\leq s<t\leq T$ we have
\[
V_1\bigl(R_H;[s,t]^2\bigr)=E
\bigl[G_H(t)-G_H(s)\bigr]^2
\leq(t-s)^{2H}.
\]
Therefore, (\ref{1sfBm3}) holds with $C_1=1$ in the case $H\in(1/2,1)$.
Now let $H\in(0,1/2)$.
Then $p=1/(2H)>1$ and the hypotheses of Theorem~\ref{SP2} hold by
Proposition~\ref{sfBm2}.
By part\vspace*{1pt} (i) of Proposition~\ref{sfBm2} and Remark~\ref{SP3}
with $L=2-2^{2H-1}$, (\ref{1sfBm3}) holds with $C_1=4-2^{2H}$ in the case
$H\in(0,1/2)$.

To prove (\ref{2sfBm3}), let $\kappa=(t_j)_{j\in[m]}$ be a partition
of $[0,T]$
and $i\in(m]$.
Due to (\ref{1sfBm3}), one can suppose that $m>1$.
First, let $H\in(1/2,1)$.
Then $p=1$.
As for fractional Brownian motion (Proposition~\ref{fBm1}), in the
present case by
part (ii) of Proposition~\ref{sfBm2} and by part (i) of
Proposition~\ref{SP1},
we have
\[
\sum_{j=1}^mV_1
\bigl(R_H;J_i^{\kappa}\times J_j^{\kappa}
\bigr) =\sum_{j=1}^mE\bigl[
\Delta_i^{\kappa}G_H\Delta_j^{\kappa}G_H
\bigr] =-\frac{1}{2} \bigl[f_T(t_i)-f_T(t_{i-1})
\bigr]+t_i^{2H}-t_{i-1}^{2H},
\]
where $f_T(t)=(T+t)^{2H}+(T-t)^{2H}$ for $t\in[0,T]$ and the last equality
is the special case of (\ref{1sfBm2}).
Since $1/2<H<1$ the function $f_T$ is increasing, and so
\[
\sum_{j=1}^mE\bigl[\Delta_i^{\kappa}G_H
\Delta_j^{\kappa}G_H\bigr] \leq
t_i^{2H}-t_{i-1}^{2H}
\leq2HT^{2H-1} (t_i-t_{i-1})
\]
by the mean value theorem.
Now let $H\in(0,1/2)$.
Then $p=1/(2H)>1$.
Again as for fractional Brownian motion (Proposition~\ref{fBm1}), in
the present case by
part (ii) of Proposition~\ref{SP1} and by parts (ii), (iii) of
Proposition~\ref{sfBm2}, we have
\[
\sum_{j\in(m]\setminus\{i\}}V_p\bigl(R_H;J_i^{\kappa}
\times J_j^{\kappa}\bigr) \leq\sum
_{j\in(m]\setminus\{i\}} \bigl|E\bigl[\Delta_i^{\kappa
}G_H
\Delta_j^{\kappa}G_H\bigr] \bigr| \leq E\bigl(
\Delta_i^{\kappa}G_H\bigr)^2.
\]
Then by part (i) of Proposition~\ref{sfBm2}, the inequality
\[
\sum_{j=1}^mV_p
\bigl(R_H;J_i^{\kappa}\times J_j^{\kappa}
\bigr) \leq\sum_{j=1}^m \bigl|E\bigl[
\Delta_i^{\kappa}G_H\Delta_j^{\kappa
}G_H
\bigr] \bigr| \leq2E\bigl(\Delta_i^{\kappa}G_H
\bigr)^2 \leq C_2\bigl(\Delta_i^{\kappa}
\bigr)^{\sfrac{1}{p}},
\]
holds with $c_2=2HT^{2H-1}$ if $2H>1$ and $C_2=4-2^{2H}$ if $2H<1$,
completing the proof.
\end{pf}

\subsection*{Bifractional Brownian motion}
Let $0<T<\infty$, $0<H< 1$ and $0<K\leq1$.
The function $C_{H,K}\dvtx [0,T]^2\to\RR$ with values
%
\begin{equation}
\label{bfBm} C_{H,K}(s,t):=2^{-K} \bigl\{
\bigl(t^{2H}+s^{2H}\bigr)^K-|t-s|^{2HK}
\bigr\},
\end{equation}
$(s,t)\in[0,T]^2$, is positive definite as shown in \cite{HandV03}.
A \emph{bifractional Brownian motion} with parameters $(H,K)$
is a mean zero Gaussian stochastic process
$B_{H,K}=\{B_{H,K}(t)\dvtx  t\in[0,T]\}$ with the covariance function
$C_{H,K}$.
When $K=1$, $B_{H,1}$ is the fractional Brownian motion $B_H$ with the Hurst
index $H\in(0,1)$.
The Gaussian process $B_{H,K}$ is a self-similar stochastic process of order
$HK\in(0,1)$, the increments are not stationary and its incremental
variance function is
\[
\sigma_{B_{H,K}}^2(s,t)=2^{1-K} \bigl[|t-s|^{2HK}-
\bigl(t^{2H}+s^{2H}\bigr)^K \bigr]
+t^{2HK}+s^{2HK}
\]
for each $s,t\geq0$.
By Proposition~3.1 of  \cite{HandV03}, for every
$s,t\geq0$,
%
\begin{equation}
\label{bfBm1} 2^{-K}|t-s|^{2HK}\leq\sigma_{B_{H,K}}^2(s,t)
\leq2^{1-K}|t-s|^{2HK}.
\end{equation}
This suggests that the incremental variance function $\sigma
_{B_{H,K}}^2$ is dominated
by a single variable function $u\mapsto \mathit{const}|u|^{2HK}$, $u\in\RR$.
A more precise property is proved next.

\begin{prop}\label{bfBm2}
Let $0<T<\infty$, $0<H< 1$, $0<K<1$ and\vspace*{1pt}
$B_{H,K}=\{B_{H,K}(t)\dvtx  t\in[0,T]\}$ be a bifractional Brownian
motion with parameters $(H,K)$.
Let $\rho_{H,K}(u):=2^{(1-K)/2}u^{HK}$ for each $u\in[0,T]$, and let
$p:=\max\{1,1/(2HK)\}$.
Then $B_{H,K}\in\break\mathcal{LSI}(\rho_{H,K}(\cdot))$.
Also, there is a finite constant $C_1$ such that the inequality
%
\begin{equation}
\label{1bfBm2} V_p\bigl(C_{H,K};[a,b]^2\bigr)
\leq C_1(b-a)^{2HK}
\end{equation}
holds for any $0\leq a<b\leq T$, and there is a finite constant $C_2$
such that the inequality
%
\begin{equation}
\label{2bfBm2} \sum_{j=1}^mV_p
\bigl(C_{H,K};J_i^{\kappa}\times J_j^{\kappa}
\bigr) \leq C_2\bigl(\Delta_i^{\kappa}
\bigr)^{1\wedge(2HK)}
\end{equation}
holds for any partition $(t_j)_{j\in[m]}$ of $[0,T]$ and for any $i\in(m]$.
\end{prop}

\begin{pf}
Concerning the property of local stationarity of increments of
$B_{H,K}$ with
the local variance function $\rho=\rho_{H,K}(\cdot)$
note that condition (A1) in Definition~\ref{LV}
holds with $L=2^{1-K}$ by (\ref{bfBm1}).
To prove condition (A2) let $\epsilon>0$.
For each $s\in[\epsilon,T)$ and $t\in(-\epsilon,T-s]$ let
\[
f_s(t):=2^{1-K} \bigl[s^{2H}+(s+t)^{2H}
\bigr]^K-s^{2HK}-(s+t)^{2HK}.
\]
Then $f_s(0)=f_s'(0)=0$ and
\[
b(s,s+h):= \biggl(\frac{\sigma_{B_{H,K}}(s,s+h)}{\rho_{H,K}(h)} \biggr)^2-1=-2^{K-1}h^{-2HK}f_s(h).
\]
Let $s\in[\epsilon,T)$ and $h\in(0,T-s]$.
By Taylor's theorem with the Lagrange remainder applied to the function
$f_s$, there exists $u=u(s,h)\in(0,h)$ such that\vspace*{-1.5pt}
\[
b(s,s+h)=-2^{K-2}h^{2(1-HK)}f_s''(u),
\]
where\vspace*{-1.5pt}
\begin{eqnarray*}
f_s''(u)&=&2^{3-K}K(K-1)H^2
\bigl[(s+u)^{2H}+s^{2H}\bigr]^{K-2}(s+u)^{2(2H-1)}
\\[-1.5pt]
&&{}+2^{2-K}KH(2H-1)\bigl[(s+u)^{2H}+s^{2H}
\bigr]^{K-1}(s+u)^{2H-2}
\\[-1.5pt]
&&{}-2HK(2HK-1) (s+u)^{2HK-2}.
\end{eqnarray*}
Then there is a finite constant $C=C(\epsilon,H,K)$ such that\vspace*{-1.5pt}
\[
\bigl|b(s,s+h)\bigr|\leq Ch^{2(1-HK)}
\]
for each $s\in[\epsilon,T)$ and $h\in(0,T-s]$.
Since $HK<1$, the preceding bound yields that condition (A2) holds,
and so $B_{H,K}\in\mathcal{LSI}(\rho_{H,K}(\cdot))$.

To prove (\ref{1bfBm2}) and (\ref{2bfBm2}), we use a decomposition
in distribution of a fractional Brownian motion $B_{HK}$ with the Hurst
index $HK$
into a linear combination of a bifractional Brownian motion $B_{H,K}$
and a Gaussian process $Y_{H,K}$ with the covariance function\vspace*{-1.5pt}
%
\begin{equation}
\label{3bfBm2} D_{H,K}(s,t):=\frac{\Gamma(1-K)}{K} \bigl[t^{2HK}+s^{2HK}-
\bigl(t^{2H} +s^{2H}\bigr)^K \bigr],
\end{equation}
$(s,t)\in[0,T]^2$, due to  \cite{LandN2009},
Proposition~1.
Letting $A:=2^{-K}K/\Gamma(1-K)$ and $B:=2^{1-K}$,
by the decomposition we have the relation\vspace*{-1.5pt}
%
\begin{equation}
\label{4bfBm2} C_{H,K}=-AD_{H,K} +BF_{HK}
\end{equation}
between the covariance functions of $B_{H,K}$, $Y_{H,K}$ and $B_{HK}$,
respectively.
For any $0\leq u<v\leq T$ and $0\leq s<t\leq T$, if $Q=[u,v]\times
[s,t]$ and
$f(r)=f_{u,v}(r):=(u^{2H}+r^{2H})^K-(v^{2H}+r^{2H})^K$ for $r\geq0$, then\vspace*{-1.5pt}
\[
\Delta^QD_{H,K}=\frac{\Gamma(1-K)}{K}\bigl[f(t)-f(s)
\bigr]>0,
\]
since $f'(r)>0$ for each $r>0$, and so $Y_{H,K}$ has positively correlated
increments.
Let $0\leq a<b\leq T$.
Since $V_p(\cdot)\leq V_1(\cdot)$,
by part (i) of Proposition~\ref{SP1}, it follows that\vspace*{-1.5pt}
\begin{eqnarray*}
V_p\bigl(C_{H,K};[a,b]^2\bigr) &\leq& A
V_1\bigl(D_{H,K};[a,b]^2\bigr)+BV_p
\bigl(F_{HK};[a,b]^2\bigr)
\\[-1.5pt]
&=& AE\bigl[Y_{H,K}(b)-Y_{H,K}(a)\bigr]^2+BV_p
\bigl(F_{HK};[a,b]^2\bigr).
\end{eqnarray*}
Using (\ref{3bfBm2}) we have\vspace*{-1.5pt}
\begin{eqnarray*}
AE\bigl[Y_{H,K}(b)-Y_{H,K}(a)\bigr]^2&=&2^{-K}
\bigl[2\bigl(b^{2H}+a^{2H}\bigr)^K-2^Kb^{2HK}
-2^Ka^{2HK} \bigr]
\\[-1.5pt]
&\leq& 2^{-K}(b-a)^{2HK}
\end{eqnarray*}
by the left inequality in (\ref{bfBm1}).
Using inequality (\ref{1fBm1}) for the fractional Brownian motion with
the Hurst
index $HK$, the first desired bound (\ref{1bfBm2})
with $C_1=5\cdot2^{-K}$ follows.

To prove the second desired bound (\ref{2bfBm2}) let
$(t_j)_{j\in[m]}$ be a partition of $[0,T]$ and let $i\in(m]$.
Again, since $Y_{H,K}$ has positively correlated increments and using
(\ref{3bfBm2})
it follows that
\begin{eqnarray*}
A\sum_{j=1}^m\bigl|E\Delta_i^{\kappa}Y_{H,K}
\Delta_j^{\kappa}Y_{H,K}\bigr| &=&AE\bigl[Y_{H,K}(t_i)-Y_{H,K}(t_{i-1})
\bigr]Y_{H,K}(T)
\\
&=&2^{-K} \bigl[t_i^{2HK}-t_{i-1}^{2HK}+
\bigl(t_{i-1}^{2H}+T^{2H}\bigr)^K -
\bigl(t_i^{2H}+T^{2H}\bigr)^K \bigr]
\\
&\leq& \lleft\{ %
\begin{array} {l@{\qquad}l} 2^{-K}(t_i-t_{i-1})^{2HK},
&\mbox{if $2HK< 1$,}
\\
2^{1-K}HKT^{2HK-1}(t_i-t_{i-1}), &
\mbox{if $2HK\geq1$.}
\\
\end{array} %
\rright.
\end{eqnarray*}
Since $V_p(\cdot)\leq V_1(\cdot)$, using (\ref{4bfBm2}), (\ref{2fBm1})
with $HK$ in place of $H$, and the preceding inequality, it follows that
\begin{eqnarray*}
&&\sum_{j=1}^mV_p
\bigl(C_{H,K};J_i^{\kappa}\times J_j^{\kappa}
\bigr)
\\
&&\quad \leq A\sum_{j=1}^m
\bigl|E\bigl[\Delta_i^{\kappa}Y_{H,K}\Delta
_j^{\kappa}Y_{H,K}\bigr] \bigr| +B\sum
_{j=1}^m \bigl|E\bigl[\Delta_i^{\kappa}B_{HK}
\Delta_j^{\kappa
}B_{HK}\bigr] \bigr|
\\
&&\quad \leq C_2(t_i-t_{i-1})^{1\wedge(2HK)},
\end{eqnarray*}
where $C_2=7\cdot2^{-K}$ if $2HK<1$
and $C_2=6HK2^{-K}T^{2HK-1}$ if $2HK\geq1$.
This completes the proof of the proposition.
\end{pf}

\section{Proof of the main result}\label{sec4}

The main result is Theorem~\ref{asconvergence} below
dealing with almost sure convergence of sums of properly
normalized powers of increments of the q.m. integral process (\ref{Y1}).
First, we prove a convergence of the mean of such sums
under less restrictive assumptions.

\begin{teo}\label{mean}
Let $r>0$ and $T>0$.
Let $X=\{X(t)\dvtx  t\in[0,T]\}$ be a mean zero Gaussian process from
the class
$\mathcal{LSI}(\rho(\cdot))$ with the covariance function $\Gamma_X$
such that for a constant $C_1$ and a number $p\geq1$ the inequality
%
\begin{equation}
\label{G7} V_p\bigl(\Gamma_X;[s,t]^2
\bigr)\leq C_1 \bigl[\rho(t-s)\bigr]^{2}
\end{equation}
holds for all $0\leq s<t\leq T$.
Let $f\in\CLW_q[0,T]$ with $q\in\CLQ_p$ and
let $(\kappa_n)$ be a sequence
of partitions $\kappa_n=(t_i^n)_{i\in[m_n]}$ of $[0,T]$ such that
$|\kappa_n|\to0$ as $n\to\infty$.
Then there exists the q.m. integral process $Y(t)=\mathrm{q.m.}\int_0^tf
\,\mathrm{d}X$, $t\in[0,T]$,
and
%
\begin{equation}
\label{1mean} \lim_{n\to\infty}\sum_{i=1}^{m_n}
\frac{E|\Delta_i^nY|^r}{[\rho
(\Delta_i^n)]^{r}} \Delta_i^n =E|\eta|^r\int
_0^T|f|^r,
\end{equation}
where $\eta$ is a standard normal random variable.
\end{teo}

\begin{pf}
Since $\rho(\cdot)$ is continuous at zero, by (A1) of Definition~\ref{LV},
it follows that $\Gamma_X$ is a continuous function.
Then the q.m. integral process $Y$ exists by Theorems \ref{integral}
and \ref{Y2}.
We shall prove (\ref{1mean}).
Since $Y$ is a Gaussian process, 
for $0\leq s<t\leq T$ we have
\[
E \bigl|Y(t)-Y(s) \bigr|^r =E|\eta|^r \biggl(E \biggl[\int
_{s}^{t}f \,\mathrm{d}X \biggr]^2
\biggr)^{r/2}.
\]
By (\ref{1Loeve1}), we have
\begin{eqnarray*}
E \biggl[\int_{s}^{t}f \,\mathrm{d}X \biggr]^2&=&
\int_{s}^{t}\int_{s}^{t}
f\otimes f \,\mathrm{d}^2\Gamma_X
\\
&=& \biggl|\int_{s}^{t}\int_{s}^{t}
\bigl[f\otimes f-f^2(s) \bigr]\, \mathrm{d}^2\Gamma_X
+f^2(s)E \bigl[X(t)-X(s) \bigr]^2 \biggr|.
\end{eqnarray*}
For $(s,t)\in[0,T]^2$ let
%
\begin{equation}
\label{G5} b(s,t):= \frac{\sigma_X^2(s,t)}{[\rho(|t-s|)]^2}-1,
\end{equation}
if $s\neq t$, and let $b(s,t):=0$ if $s=t$.
Then
\begin{eqnarray*}
R_n&:=&\bigl(E|\eta|^r\bigr)^{-1}\sum
_{i=1}^{m_n}\frac{E|\Delta_i^nY|^r
}{[\rho(\Delta_i^n)]^{r}}\Delta_i^n
\\
&=&\sum_{i=1}^{m_n} \biggl|\frac{1}{[\rho(\Delta_i^n)]^{2}}
\int_{t_{i-1}^n}^{t_i^n}\int_{t_{i-1}^n}^{t_i^n}
\bigl[f\otimes f-f^2\bigl(t_{i-1}^n\bigr)
\bigr] \,\mathrm{d}^2\Gamma_X
 +f^2\bigl(t_{i-1}^n\bigr) \bigl[1 + b
\bigl(t_{i-1}^n,t_i^n\bigr) \bigr]
\biggr|^{r/2}\Delta_i^n
\end{eqnarray*}
for each $n$.
Also for each integer $n\geq1$, let
\[
T_n:=\sum_{i=1}^{m_n} \bigl
\{f^2\bigl(t_{i-1}^n\bigr) \bigr
\}^{r/2}\Delta_i^n, \qquad U_n:=\sum
_{i=1}^{m_n} \bigl\{f^2
\bigl(t_{i-1}^n\bigr)\bigl|b\bigl(t_{i-1}^n,t_i^n
\bigr)\bigr| \bigr\}^{r/2}\Delta_i^n
\]
and
\[
W_n:=\sum_{i=1}^{m_n} \biggl\{
\frac{1}{[\rho(\Delta_i^n)]^{2}} \biggl|\int_{t_{i-1}^n}^{t_i^n}\int
_{t_{i-1}^n}^{t_i^n} \bigl[f\otimes f-f^2
\bigl(t_{i-1}^n\bigr) \bigr] \,\mathrm{d}^2
\Gamma_X \biggr| \biggr\} ^{r/2}\Delta_i^n.
\]
If $r<2$ then using the inequality $||A|^{r/2}-|B|^{r/2}|\leq
|A-B|^{r/2}$ it follows that
%
\begin{equation}
\label{4mean} |R_n-T_n|\leq U_n+W_n
\end{equation}
for each $n$.
If $r\geq2$, then
using the Minkowskii inequality for weighted sums, it follows that
%
\begin{equation}
\label{2mean} \bigl|R_n^{2/r}-T_n^{2/r}\bigr|
\leq U_n^{2/r}+W_n^{2/r}
\end{equation}
for each $n$.
Recall that the mesh $|\kappa_n|\to0$ as $n\to\infty$.
Therefore since $f$ is regulated, and so $|f|^r$ is Riemann integrable,
it follows that
%
\begin{equation}
\label{3mean} \lim_{n\to\infty}T_n =\lim
_{n\to\infty}\sum_{i=1}^{m_n} \bigl|f
\bigl(t_{i-1}^n\bigr) \bigr|^{r} \Delta_i^n
=\int_0^T|f|^r.
\end{equation}
We will show that $U_n$ and $W_n$ tend to zero as $n\to\infty$.
Assuming this, by (\ref{4mean}) if $r<2$, by (\ref{2mean}) if $r\geq
2$ and (\ref{3mean}),
the conclusion (\ref{1mean}) follows.

To prove convergence of $U_n$ let $\epsilon>0$.
Recalling notation (\ref{G5}) and using condition (A1) of Definition~\ref{LV},
we have
\[
\bigl|b\bigl(t_{i-1}^n,t_i^n\bigr)\bigr|
\leq\frac{E [\Delta_i^nX]^2}{[\rho(\Delta_i^n)]^2} +1\leq L^2+1
\]
for each $n\in\NN$ and $i\in(m_n]$.
For each $\delta>0$ letting
%
\begin{equation}
\label{G4} \phi_{\epsilon}(\delta):=\sup \bigl\{\bigl|b(s,s+h)\bigr|\dvtx  s\in
[\epsilon,T), h\in\bigl(0,\delta\wedge(T-s)\bigr] \bigr\}
\end{equation}
it follows that
\[
\sum_{i\colon t_{i-1}^n\geq\epsilon} \bigl\{f^2
\bigl(t_{i-1}^n\bigr) \bigl|b\bigl(t_{i-1}^n,t_i^n
\bigr)\bigr| \bigr\}^{r/2}\Delta_i^n \leq \bigl\{
\phi_{\epsilon}\bigl(|\kappa_n|\bigr) \bigr\}^{r/2}\sum
_{i=1}^{m_n} \bigl|f\bigl(t_{i-1}^n
\bigr)\bigr|^r\Delta_i^n
\]
for all $n\in\NN$.
Then for each $n\in\NN$
we have
\begin{eqnarray*}
U_n&\leq& \biggl(\sum_{i\colon t_{i-1}^n\geq\epsilon} + \sum
_{i\colon t_{i-1}^n<\epsilon} \biggr) \bigl\{f^2
\bigl(t_{i-1}^n\bigr)\bigl|b\bigl(t_{i-1}^n,t_i^n
\bigr)\bigr| \bigr\}^{r/2}\Delta_i^n
\\
&\leq& \bigl\{\phi_{\epsilon}\bigl(|\kappa_n|\bigr) \bigr
\}^{r/2}\sum_{i=1}^{m_n} \bigl|f
\bigl(t_{i-1}^n\bigr)\bigr|^r\Delta_i^n
+\bigl(\epsilon+|\kappa_n|\bigr)\|f\|_{\sup}^r
\bigl(L^2+1\bigr)^{r/2}.
\end{eqnarray*}
By conditions (A1) and (A2) of Definition~\ref{LV}, and since the
Riemann sums are bounded
as $|\kappa_n|\to0$ with $n\to\infty$,
$U_n$ tends to zero as $n\to\infty$.

We prove convergence of $W_n$ first assuming that $p=1$.
In this case, by (\ref{1DRS}) and (\ref{G7}), we have
%
\begin{eqnarray}\label{mean1}
W_n&\leq& \bigl(2\|f\|_{\sup} \bigr)^{\sfrac{r}{2}}\sum
_{i=1}^{m_n} \biggl\{\frac{V_1(\Gamma_X; [t_{i-1}^n,t_i^n]^2)}{[\rho(\Delta_i^n)]^{2}} \osc\bigl(f;
\bigl[t_{i-1}^n,t_i^n\bigr]\bigr)
\biggr\}^{\sfrac{r}{2}} \Delta_i^n
\nonumber
\\[-8pt]\\[-8pt]
&\leq& \bigl(2C_1\|f\|_{\sup} \bigr)^{\sfrac{r}{2}}\sum
_{i=1}^{m_n} \bigl\{ \osc\bigl(f;\bigl[t_{i-1}^n,t_i^n
\bigr]\bigr) \bigr\}^{\sfrac{r}{2}} \Delta_i^n.\nonumber
\end{eqnarray}
Let $\epsilon>0$.
Since $f$ is a regulated function there is a partition $\{s_j\}_{j=0}^k$
of $[0,T]$\vspace*{2pt} such that $\osc(f;(s_{j-1},s_j))<\epsilon$ for each $j\in(k]$
by Theorem~2.1 in \cite{DandN2010}.
Since $|\kappa_n|\to0$ as $n\to\infty$ there is an $N\in\NN$
such that $|\kappa_n|<\epsilon/(2k)$ for each $n\geq N$.
For each $n$ let $J_n$ be the set of indices $i\in(m_n]$
such that $s_j\in[t_{i-1}^n,t_i^n]$ for some $j\in[k]$
and let $J_n^c:=(m_n]\setminus J_n$.
Then the cardinality of $J_n$ does not exceed $2k$,
and continuing (\ref{mean1}), we have for each $n\geq N$
\begin{eqnarray*}
W_n &\leq& \bigl(4C_1\|f\|_{\sup}^2
\bigr)^{\sfrac{r}{2}}\sum_{i\in
J_n}\Delta_i^n
+ \bigl(2\epsilon C_1\|f\|_{\sup} \bigr)^{\sfrac{r}{2}}\sum
_{i\in
J_n^c}\Delta_i^n
\\
&\leq& \epsilon \bigl(4C_1\|f\|_{\sup}^2
\bigr)^{\sfrac{r}{2}} + \bigl(2\epsilon C_1\|f\|_{\sup}
\bigr)^{\sfrac{r}{2}}T,
\end{eqnarray*}
since the mesh $|\kappa_n|<\epsilon/(2k)$.

Now suppose that $p>1$. Let $q'>q$ be such that $\frac{1}{p}+\frac{1}{q}>
\frac{1}{p}+\frac{1}{q'}>1$.
By (\ref{1LL2}), we have
\begin{eqnarray*}
W_n&\leq& \bigl(K_{p,q'}\|f\|_{[q']}\bigr)^{\sfrac{r}{2}}
\sum_{i=1}^{m_n} \biggl\{ \frac{V_p(\Gamma_X;[t_{i-1}^n,t_i^n]^2)}{[\rho(\Delta_i^n)]^{2}}
V_{q'}\bigl(f;\bigl[t_{i-1}^n,t_i^n
\bigr]\bigr) \biggr\}^{\sfrac{r}{2}}\Delta_i^n
\\
&\leq& \bigl(K_{p,q'}\|f\|_{[q']} C_1
V_q(f)^{\sfrac{q}{q'}} \bigr)^{\sfrac{r}{2}} \sum
_{i=1}^{m_n} \bigl\{\osc\bigl(f;\bigl[t_{i-1}^n,t_i^n
\bigr]\bigr)^{1-\sfrac
{q}{q'}} \bigr\}^{\sfrac{r}{2}} \Delta_i^n.
\end{eqnarray*}
Since a function of bounded $q$-variation is regulated the
arguments used in the preceding case $p=1$ apply and show that
$W_n$ tends to zero as $n\to\infty$.
The theorem is proved.
\end{pf}

In the case $f\equiv1$, we have the following conclusion.

\begin{cor}\label{meanX}
Let $r>0$ and $T>0$.
Let $X=\{X(t)\dvtx t\in[0,T]\}$ be a mean zero Gaussian process from
the class
$\mathcal{LSI}(\rho(\cdot))$, and let $(\kappa_n)$ be a sequence
of partitions $\kappa_n=(t_i^n)_{i\in[m_n]}$ of $[0,T]$ such that
$|\kappa_n|\to0$ as $n\to\infty$.
Then
\[
\lim_{n\to\infty}\sum_{i=1}^{m_n}
\frac{E|\Delta_i^nX|^r}{[\rho
(\Delta_i^n)]^{r}} \Delta_i^n =E|\eta|^rT,
\]
where $\eta$ is a standard normal random variable.
\end{cor}

\begin{pf}
In the proof of Theorem~\ref{mean} taking $f\equiv1$ it follows that
for each
$n\geq1$, in the present case we have $T_n=T$, $W_n=0$,
\[
R_n=\sum_{i=1}^{m_n}\bigl[1+b
\bigl(t_{i-1}^n,t_i^n\bigr)
\bigr]^{r/2}\Delta_i^n \quad \mbox{and}\quad
U_n=\sum_{i=1}^{m_n}\bigl[b
\bigl(t_{i-1}^n,t_i^n\bigr)
\bigr]^{r/2}\Delta_i^n.
\]
The argument used in the proof of Theorem~\ref{mean} gives that
$U_n\to0$ as
$n\to\infty$, and so $R_n\to T$ as $n\to\infty$, proving the corollary.
\end{pf}

%
%

Next, is the main result.

\begin{teo}\label{asconvergence}
Let $T>0$, let $\rho\in R[0,T]$ be such that
$\gamma_{\ast}(\rho)=\gamma$ for some $\gamma\in(0,1)$,
let $p:=\max\{1,1/(2\gamma)\}$, and let $1<r<2/\max\{(2\gamma-1),0\}$.
Let $X$ be a mean zero Gaussian process from the class $\mathcal
{LSI}(\rho(\cdot))$
with the covariance function $\Gamma_X$.
Suppose that there is a constant $C_1$ such that
\textup{(\ref{G7})} holds for all $0\leq s<t\leq T$,
and there is a constant $C_2$ such that the inequality
%
\begin{equation}
\label{3asconvergence} \sum_{j=1}^{m}V_p
\bigl(\Gamma_X;J_i^{\kappa}\times
J_j^{\kappa}\bigr) \leq C_2 \bigl(
\Delta_i^{\kappa}\bigr)^{1\wedge(2\gamma)}
\end{equation}
holds for each partition $\kappa=(t_j)_{j\in[m]}$ of $[0,T]$ and each
$i\in(m]$.
Let $f\in\CLW_q[0,T]$ with $q\in\CLQ_p$, and
let $(\kappa_n)$ be a sequence of partitions $\kappa_n=(t_i^n)_{i\in[m_n]}$
of $[0,T]$ such that
%
\begin{equation}
\label{2asconvergence} \sup\Bigl\{\alpha\dvtx \lim_{n\to\infty}|
\kappa_n|^{\alpha}\log n=0\Bigr\} = \biggl(1\wedge
\frac{2}{r} \biggr)+ \bigl(0\wedge(1-2\gamma) \bigr).
\end{equation}
Then there exists the q.m. integral process $Y(t)=\mathrm{q.m.}\int_0^tf \,\mathrm{d}X$,
$t\in[0,T]$, and with probability one
%
\begin{equation}
\label{4asconvergence} \lim_{n\to\infty}\sum_{i=1}^{m_n}
\frac{|\Delta_i^nY|^r}{[\rho
(\Delta_i^n)]^{r}} \Delta_i^n=E|\eta|^r\int
_0^T|f|^r,
\end{equation}
where $\eta$ is a standard normal random variable.
\end{teo}

\begin{rem}\label{mesh}
The right side of (\ref{2asconvergence}) is less than or equal to $1$.
Also it is positive for any $1<r<\infty$ if $\gamma\leq1/2$, and
for any $1<r<2/(2\gamma-1)$ if $\gamma>1/2$.
It follows from the proof of the theorem that
if the local variance $\rho(u)=u^{\gamma}$ then the hypothesis (\ref
{2asconvergence})
can be replaced by the following one
\[
\lim_{n\to\infty}|\kappa_n|^{(1\wedge\sfrac{2}{r})+(0\wedge
(1-2\gamma)) }\log n=0.
\]
It is known that this condition with $r=2$ is best possible
(\cite{V1974} and \cite
{MandR1992}, Theorem~2.6).
\end{rem}

\begin{pf*}{Proof of Theorem~\ref{asconvergence}}
The q.m. integral process $Y(t)=\mathrm{q.m.}\int_0^tf \,\mathrm{d}X$, $t\in[0,T]$, exists
due to reasons stated in the proof of Theorem~\ref{mean}.
For each $n\geq1$, let
\[
Z_n:= \Biggl(\sum_{i=1}^{m_n}c_{i,n}\bigl|
\Delta_i^nY\bigr|^r \Biggr)^{1/r},
\qquad \mbox{where } c_{i,n}:=\bigl[\rho\bigl(\Delta_i^n
\bigr)\bigr]^{-r}\Delta_i^n.
\]
Denoting the median of a real random variable $Z$ by $\med(Z)$,
by Lemma~2.2 of \cite{MandR1992}, for each
$\epsilon>0$
%
\begin{equation}
\label{1asconvergence} \Pr \bigl( \bigl\{\bigl|Z_n-\med(Z_n)\bigr|>\epsilon
\bigr\} \bigr)\leq2\exp \biggl\{ -\frac{\epsilon^2}{2\sigma_n^2} \biggr\},
\end{equation}
where
\[
\sigma_n^2:=\sup \Biggl\{E \Biggl(\sum
_{i=1}^{m_n}b_ic_{i,n}^{1/r}
\Delta _i^nY \Biggr)^2 \dvtx (b_i)_{i\in(m_n]}
\in\RR^{m_n}, \sum_{i=1}^{m_n}|b_i|^{r'}
\leq1 \Biggr\}
\]
and $1/r+1/r'=1$.
For each $n\geq1$ and $i\in(m_n]$, by (\ref{1Loeve1}) we have\vspace*{-1pt}
\[
M_{i,n}:=\sum_{j=1}^{m_n} \bigl|E
\bigl(\Delta_i^nY\Delta_j^nY
\bigr) \bigr| =\sum_{j=1}^{m_n} \biggl|\int
_{t_{i-1}^n}^{t_i^n}\int_{t_{j-1}^n}^{t_j^n}
f\otimes f\, \mathrm{d}^2\Gamma_X \biggr|.
\]
For each $n\geq1$ and
each vector $(b_i)\in\RR^{m_n}$, by Lemma~2.2 of  \cite{QMS1996}\vspace*{-1pt}
\begin{eqnarray*}
E \Biggl(\sum_{i=1}^{m_n}b_ic_{i,n}^{1/r}
\Delta_i^nY \Biggr)^2& =& \Biggl|\sum
_{i=1}^{m_n}\sum_{j=1}^{m_n}b_ib_j(c_{i,n}c_{j,n})^{\sfrac{1}{r}}
E \bigl(\Delta_i^nY \Delta_j^nY
\bigr) \Biggr|
\\
&\leq&\sum_{i=1}^{m_n}b_i^2c_{i,n}^{\sfrac{2}{r}}M_{i,n}
\\
&\leq&\lleft\{ %
\begin{array} {l@{\qquad}l} \displaystyle \max_{1\leq i\leq m_n}c_{i,n}^{\sfrac{2}{r}}M_{i,n}
\Biggl(\displaystyle \sum_{j=1}^{m_n} |b_j|^{r'}
\Biggr)^{\sfrac{2}{r'}},&\mbox{if $2\leq r<\infty$},
\\
\Biggl(\displaystyle \sum_{i=1}^{m_n}c_{i,n}^{\afrac{2}{2-r}}M_{i,n}^{\afrac{r}{2-r}}
\Biggr)^{\vfrac{2-r}{r}} \Biggl(\displaystyle \sum_{j=1}^{m_n}
|b_j|^{r'} \Biggr)^{\sfrac{2}{r'}}, &\mbox{if $1<r<2$}.
\end{array} %
\rright.
\end{eqnarray*}
It then follows that for each $n\geq1$\vspace*{-1pt}
\[
\sigma_n^2\leq\lleft\{ %
\begin{array} {l@{\qquad}l}
\displaystyle \max_{1\leq i\leq m_n} \biggl(\displaystyle \frac{\Delta_i^n}{[\rho(\Delta
_i^n)]^r} \biggr)^{\sfrac{2}{r}}
M_{i,n} ,&\mbox{if $2\leq r<\infty$},
\\
\Biggl(\displaystyle \sum_{i=1}^{m_n} \biggl(
\displaystyle \frac{\Delta_i^n}{[\rho(\Delta
_i^n)]^r} \biggr)^{\afrac{2}{2-r}} M_{i,n}^{\afrac{r}{2-r}}
\Biggr)^{\vfrac{2-r}{r}}, &\mbox{if $1<r<2$}. \end{array} %
\rright.
\]
By (\ref{2LL2}) if $p>1$ and by (\ref{2DRS}) if $p=1$, and then by
(\ref{3asconvergence}), for each $i$\vspace*{-1pt}
\[
M_{i,n}\leq K_{p,q}\|f\|_{[q]}^2\sum
_{j=1}^{m_n}V_p\bigl(\Gamma
_X;J_{i}^{\kappa_n}\times J_{j}^{\kappa_n}
\bigr) \leq C_2K_{p,q}\|f\|_{[q]}^2
\bigl[\Delta_i^n\bigr]^{\sfrac{1}{p}},
\]
where $K_{1,\infty}:=1$ and $\|f\|_{[\infty]}:=\|f\|_{\sup}$.
By the hypothesis on the local variance $\rho$, we have\vspace*{-1pt}
\[
\gamma_{\ast}(\rho)
=\inf \biggl\{\alpha>0\dvtx
\sup_{u> 0}\frac{u^{\alpha}}{\rho
(u)}<\infty \biggr\} =\gamma.
\]
In the case $p> 1$, we have $1-2\gamma> 0$, and
so for each $\delta\in(0,1/r)$,\vspace*{-1pt}
\begin{eqnarray*}
\sigma_n^2&\leq& 
C_2K_{p,q}\|f\|_{[q]}^2\lleft\{
\begin{array} {l@{\qquad}l} |\kappa_n|^{\sfrac{2}{r}-2\delta} \displaystyle \max
_{1\leq i\leq m_n} \biggl(\displaystyle \frac{(\Delta_i^n)^{\gamma+\delta
}}{\rho(\Delta_i^n)} \biggr)^{2} ,&
\mbox{if $2\leq r<\infty$},
\\
T^{\vfrac{2-r}{r}}|\kappa_n|^{1-2\delta} \displaystyle \max_{1\leq i\leq m_n}
\biggl(\displaystyle \frac{(\Delta_i^n)^{\gamma+\delta
}}{\rho(\Delta_i^n)} \biggr)^{2} , &\mbox{if $1<r<2$},
\end{array} %
\rright.
\\
&=&\mathrm{o}\bigl(1/(\log n)\bigr),
\end{eqnarray*}
as $n\to\infty$ by the hypothesis (\ref{2asconvergence}).
In the case $p= 1$, we have $1-2\gamma\leq0$, and so for each
$\delta>0$,
\begin{eqnarray*}
\sigma_n^2&\leq& C_2\|f\|_{\sup}^2
\lleft\{ %
\begin{array} {l@{\qquad}l} |\kappa_n|^{\sfrac{2}{r}+1-2\gamma-2\delta} \displaystyle \max
_{1\leq i\leq m_n} \biggl(\displaystyle \frac{(\Delta_i^n)^{\gamma+\delta
}}{\rho(\Delta_i^n)} \biggr)^{2} ,&
\mbox{if $2\leq r<\infty$},
\\
T^{\vfrac{2-r}{r}}|\kappa_n|^{2-2\gamma-2\delta}\displaystyle  \max_{1\leq i\leq m_n}
\biggl(\displaystyle \frac{(\Delta_i^n)^{\gamma+\delta
}}{\rho(\Delta_i^n)} \biggr)^{2} , &\mbox{if $1<r<2$},
\end{array} %
\rright.
\\
&=&\mathrm{o}\bigl(1/(\log n)\bigr),
\end{eqnarray*}
as $n\to\infty$ by the hypothesis (\ref{2asconvergence}).
By (\ref{1asconvergence}) and Borel--Cantelli lemma it then follows that
\[
\lim_{n\to\infty}\bigl|Z_n-\med(Z_n)\bigr|=0
\]
with probability one.
Using our Theorem~\ref{mean} and the argument of  \cite{MandR1992}, Theorem~2.3,
it follows that (\ref{4asconvergence}) holds with probability one.
\end{pf*}

In the case $f\equiv1$, we have the following conclusion.

%
\begin{cor}\label{asconvergenceX}
Let $T>0$, let $\rho\in R[0,T]$ be such that
$\gamma_{\ast}(\rho)=\gamma$ for some $\gamma\in(0,1)$,
and let $1<r<2/\max\{(2\gamma-1),0\}$.
Let $X=\{X(t)\dvtx t\in[0,T]\}$ be a mean zero Gaussian process from
the class
$\mathcal{LSI}(\rho(\cdot))$.
Suppose there is a constant $C_2$ such that the inequality
%
\begin{equation}
\label{1asconvergenceX} \sum_{j=1}^{m} \bigl|E\bigl[
\Delta_i^{\kappa}X\Delta_j^{\kappa}X\bigr] \bigr|
\leq C_2 \bigl(\Delta_i^{\kappa}
\bigr)^{1\wedge(2\gamma)}
\end{equation}
holds for each partition $\kappa=(t_j)_{j\in[m]}$ of $[0,T]$ and each
$i\in(m]$.
Let $(\kappa_n)$ be a sequence of partitions $\kappa_n=(t_i^n)_{i\in[m_n]}$
of $[0,T]$ such that \textup{(\ref{2asconvergence})} holds.
Then with probability one
%
\begin{equation}
\label{2asconvergenceX} \lim_{n\to\infty}\sum_{i=1}^{m_n}
\frac{|\Delta_i^nX|^r}{[\rho
(\Delta_i^n)]^{r}} \Delta_i^n=E|\eta|^rT,
\end{equation}
where $\eta$ is a standard normal random variable.
\end{cor}

\begin{pf}
The proof is the same as of Theorem~\ref{asconvergence} except that now
Corollary~\ref{meanX} is used in place of Theorem~\ref{mean} and the bound
(\ref{1asconvergenceX}) is used in place of (\ref{3asconvergence}).
\end{pf}

\section*{Acknowledgements}
This research was funded by Grants (No. MIP-66/2010 and No.
MIP-053/2012) from the Research Council of Lithuania.



\printhistory

\end{document}